\documentclass[ijoc,sglanonrev]{informs4}
\usepackage{eqndefns-left} 
\RequirePackage{tgtermes}
\RequirePackage{newtxtext}
\RequirePackage{newtxmath}
\RequirePackage{bm}
\RequirePackage{endnotes}
\usepackage{amsmath}
\usepackage{amssymb}

\usepackage{xcolor}

\OneAndAHalfSpacedXII 

\usepackage{algorithm}
\usepackage{algpseudocode}
\usepackage{tikz}
\usepackage{tikz}
\usetikzlibrary{arrows.meta} 
\usetikzlibrary{positioning} 

\usepackage{natbib}
 \bibpunct[, ]{(}{)}{,}{a}{}{,}%
 \def\bibfont{\small}%
\EquationsNumberedThrough    

\TheoremsNumberedThrough     
\ECRepeatTheorems  %

\MANUSCRIPTNO{}

\begin{document}


\RUNAUTHOR{Adegbindin, Dietrich, and Pender}

\RUNTITLE{LP-State Steering for MIP Performance Variability}

\TITLE{Mitigating Performance Variability in Mixed-Integer Programming Solvers via LP State Steering}

\ARTICLEAUTHORS{%
\AUTHOR{Mourchid Adegbindin, Brenda Dietrich, Jamol Pender}
\AFF{School of Operations Research and Information Engineering,
Cornell University, Ithaca, NY, 14853,\\ \EMAIL{ma757@cornell.edu}, \EMAIL{bld34@cornell.edu}, \EMAIL{jjp274@cornell.edu}}

} 

\ABSTRACT{%
Modern mixed-integer programming (MIP) solvers repeatedly solve linear programming (LP) relaxations, but the LP state transferred from the simplex algorithm to the branch-and-bound tree is typically treated as an artifact of perturbation, presolve history, and numerical trajectory. This is harmless in pure linear 
programming, where alternative optimal bases certify the same bound. In MIP, however, those bases can expose distinct fractional structures, active rows, reduced costs, branching candidates, and reoptimization states. The inherited LP state is therefore an initial condition for the integer search. We exploit this initial condition to mitigate computational performance variability. The 
proposed framework generates alternative LP states on or near degenerate optimal faces, evaluates them within an auxiliary solver environment, and either selects a single state or executes a portfolio of candidate states in parallel. A handcrafted synthetic instance isolates the relationship between the initial LP state, the fractional profile, the branching trajectory, 
and the resulting node-count distribution: root-state steering reduces the weighted variability score~(WVS) by 67\% over 1,000 paired seeds and produces no node-count regressions relative to default SCIP. On a 20-instance pilot set, high-quality LP states are consistently present among generated candidates and frequently appear within the selector's top five options, which justifies 
the use of small-portfolio policies. On a separate 130-instance MIPLIB 2017 validation set with up to 50 paired seeds per instance, the top-five policy improves solve completion from 3,792 to 3,894 runs, reduces the shifted geometric mean of the recorded node count from 2,574.7 to 1,198.0, and lowers 
the median WVS by 9\% on instances reliably solved by the baseline. These 
results demonstrate that the branch-and-bound phase inherits a specific LP state 
rather than merely an LP bound, and that the active management of this state 
systematically mitigates computational performance variability.
}%




 
\maketitle


\section{Introduction}
\label{sec:introduction}

\subsection{Motivation: Performance Variability and LP-State Inheritance}
\label{sec:intro_motivation}

Modern mixed-integer programming (MIP) solvers are highly sophisticated engines for general-purpose optimization. Yet their performance can vary widely under minor perturbations, such as changes to the random seed, row or column ordering, presolve choices, or the computing environment. An instance that solves quickly under one configuration may require an orders-of-magnitude larger branch-and-bound tree under another, even though the mathematical  formulation and optimal objective value are unchanged.

This variability complicates solver profiling, makes empirical comparisons fragile, and introduces operational risk when optimization models are embedded in automated, time-limited decision pipelines. In such settings, tail behavior matters alongside median performance because a small number of delayed solves can determine whether a computation finishes within its operational budget.

A common interpretation is that this variance is an unavoidable consequence of combinatorial search. That view is incomplete. Branch-and-bound is guided by a sequence of LP relaxations, and degenerate relaxations can admit multiple optimal simplex bases. These bases certify the same LP bound, but they need not be equivalent for the subsequent MIP search. They can expose different fractional solutions, active rows, reduced costs, branching candidates, and reoptimization states. Thus, identical LP bounds can lead to different downstream search behavior depending on the state returned by simplex.

This observation is the starting point of this work. If multiple LP-optimal or near-optimal states are available, and if the inherited state can alter the subsequent search trajectory, then the LP state transferred from simplex to branch-and-bound is an algorithmic degree of freedom. LP optimality determines the bound, but it does not uniquely determine the state from which the integer search proceeds.

\subsection{Central Thesis and Contributions}
\label{sec:intro_contributions}

This paper treats LP-state selection as a solver-internal control variable for performance variability. In degenerate relaxations, multiple LP-optimal or near-optimal states can preserve essentially the same relaxation bound while presenting different branch-relevant information to the subsequent MIP search. We develop a basis-steering framework that generates alternative LP states using reduced-cost and LP-state information, evaluates them outside the live branch-and-bound trajectory, and commits a selected state only after compatibility checks. The solver's native branching rule is left unchanged. Depending on the resource budget, the framework can commit a single selected state or evaluate a small selected portfolio concurrently.

The contributions are fourfold. First, we give a geometric account of why LP-equivalent states can be search-non-equivalent and isolate this effect on a handcrafted synthetic instance. Second, we develop a reduced-cost-based state generator together with a state-consistent handoff that allows a selected LP state to be carried into the subsequent branch-and-bound search. Third, on a 20-instance pilot set, a best-of-\(K\) audit separates candidate availability from candidate selection and shows that useful states are often present among the generated candidates; a frozen rank-based selector and small selected portfolios recover part of this opportunity. Fourth, we evaluate the frozen policies on a separate 130-instance MIPLIB 2017 validation set, where the top-five policy increases the number of solved runs by 102, reduces the shifted geometric mean of the recorded node count by more than half, and lowers median WVS by 9\% on baseline-tractable instances.

\subsection{Relationship to Prior Work}
\label{sec:intro_prior_work}

Performance variability in mixed-integer programming has been documented for
many years. Changes in formulation, row or column ordering, randomization, and
internal solver paths can produce large differences in computational effort even
when the underlying optimization problem is unchanged
\citep{danna2008, KochEtAl2011_MIPLIB2010, LodiTramontani2013, halSortModel}.
Diversification methods address this sensitivity through restarts, portfolios,
concurrent execution, or sampling of alternative solver trajectories
\citep{gomes2000}. Learning-based methods instead use problem features or solver
state to configure algorithms or guide decisions such as branching
\citep{bengio2021machine}. The method studied here does not replace
the branching rule; it changes the LP state presented to that rule.

Degeneracy-aware manipulation of LP optima has also been used within MIP
solvers. Methods described by Achterberg construct alternative
optimal LP solutions using reduced-cost information and modified auxiliary
objectives, with applications including fractionality and cut selection
\citep{achterberg2013lp}. Related integrality-aware simplex procedures exploit
zero reduced costs to move along an optimal face toward LP solutions with
properties considered favorable for subsequent MIP processing. Miltenberger
\citep{miltenberger2016} develops this idea as LP solution polishing:
additional zero-reduced-cost pivots alter an optimal basis to decrease or
increase the fractionality of integer variables. These experiments show that
the choice of LP optimum can affect root behavior and subsequent node counts,
although no reduction in performance variability was observed.

The closest work on variability itself is Fischetti et
al.~\citep{fischetti2016random}. Their \texttt{ksample} procedure processes
several diversified root nodes, collects the cutting planes and feasible
solutions generated by those runs, and supplies the aggregated information to
a subsequent branch-and-cut run. In their experimental implementation,
diversification is induced through the CPLEX random seed, which affects several
solver components in addition to LP degeneracy. The method substantially
improves computational effort, but the authors report mixed results for
stability of the complete enumeration and identify stabilization of the
branching tree as an open problem.

The present work treats alternative LP states as candidates for explicit
selection. Candidate states are generated and evaluated outside the live tree,
and the selected state itself is handed to the subsequent branch-and-bound
search rather than using alternative root runs only as sources of cuts or
primal solutions. The empirical study separates candidate availability from
policy recovery and then evaluates, over paired runs, how the resulting
state-selection policies affect variability in the subsequent branch-and-bound
search.

\subsection{Paper Organization}
\label{sec:intro_organization}

The paper is organized as follows. Section~\ref{sec:geometry} develops the geometry of degenerate LP states. Section~\ref{sec:framework} presents the basis-steering framework. Section~\ref{sec:experiments} describes the experimental design, endpoints, the WVS metric, the synthetic instance, the pilot set, the validation protocol, and the resource accounting. Section~\ref{sec:results} reports the computational results, including the synthetic experiment, the 20-instance best-of-\(K\) pilot audit, selector and top-five ablations, and the 130-instance validation experiment. Section~\ref{sec:discussion} discusses solver-design implications, threats to validity, and limitations.


\section{Geometry of Degenerate LP Optimal Faces}
\label{sec:geometry}

\subsection{Reduced Costs and Objective Invariance}
\label{sec:objective_invariance}

Consider the linear programming relaxation solved at a node of the branch-and-bound tree:
\begin{equation}
    z^*
    =
    \min
    \left\{
        c^T x
        \;\middle|\;
        Ax=b,\;
        \ell \leq x \leq u
    \right\},
\end{equation}
where $A \in \mathbb{R}^{m \times n}$, $c \in \mathbb{R}^{n}$,
$\ell \in (\mathbb{R}\cup\{-\infty\})^n$, and $u \in (\mathbb{R}\cup\{+\infty\})^n$. The primal optimal solution set is the face
\begin{equation}
    \mathcal{F}^*
    =
    \left\{
        x \in \mathbb{R}^{n}
        \;\middle|\;
        Ax=b,\;
        \ell \leq x \leq u,\;
        c^T x = z^*
    \right\}.
\end{equation}

A simplex solve certifies optimality once it reaches an optimal basis. It is not designed to compare LP-optimal bases according to their downstream effect on integer search. When this basis choice is unique, or when alternative bases expose the same solver-facing information, the distinction is immaterial. In degenerate relaxations, however, the same LP objective value may be represented by multiple optimal bases and, in some cases, by multiple primal vertices on \(\mathcal{F}^*\). We refer to the basis, primal solution, reduced costs, active rows, and branching-visible LP information collectively as the LP state. Such states are equivalent for the continuous relaxation, but they may differ in fractional-variable structure, active-row patterns, basis conditioning, branching candidates, and reoptimization behavior after branching.

Let \(B\in\mathbb{R}^{m\times m}\) denote the basis matrix associated with an optimal basic feasible solution, and let \(N\in\mathbb{R}^{m\times(n-m)}\) denote the corresponding nonbasic columns. After reordering variables as $x=(x_B,x_N)$, the equality system gives
\begin{equation}
    x_B
    =
    B^{-1}b - B^{-1}N x_N .
    \label{eq:basic_solution}
\end{equation}
Let $c_B$ and $c_N$ denote the corresponding cost vectors. The dual multiplier associated with basis $B$ is $y^T = c_B^T B^{-1}$, and the reduced cost vector for the nonbasic columns is
\begin{equation}
    \bar c_N^T
    =
    c_N^T - c_B^T B^{-1}N .
    \label{eq:reduced_cost_vector}
\end{equation}
Equivalently, for an individual nonbasic column $j$,
\begin{equation}
    \bar c_j
    =
    c_j - c_B^T B^{-1} A_j,
    \label{eq:reduced_cost_single}
\end{equation}
where \(A_j\) denotes column \(j\) of \(A\).

We define the local reduced-cost degeneracy set-
\begin{equation}
    V_{\mathrm{deg}}
    =
    \left\{
        j \in \mathcal{N}
        \;\middle|\;
        |\bar c_j| \leq \varepsilon_{\mathrm{rc}}
    \right\},
    \label{eq:vdeg}
\end{equation}
where \(\mathcal{N}\) denotes the index set of nonbasic variables for the current basis and \(\varepsilon_{\mathrm{rc}}>0\) is a numerical zero tolerance. The cardinality $|V_{\mathrm{deg}}|$ is not, by itself, the dimension of $\mathcal{F}^*$. It is a local, computable proxy for the number of nonbasic columns that are potential zero-slope entering candidates under the current basis.

\begin{proposition}[Objective invariance along zero-reduced-cost directions]
\label{prop:objective_invariance}
Let \(B\) be a simplex basis associated with an optimal basic feasible solution of the LP relaxation. Consider a nonbasic variable $x_j$, $j \in \mathcal{N}$, whose reduced cost satisfies $|\bar c_j| \leq \varepsilon_{\mathrm{rc}}$. Suppose that $x_j$ moves from its current bound in a feasible simplex direction with step length $\theta \geq 0$, and that the basic variables are updated to preserve $Ax=b$. Then $|c^T\Delta x|\leq \theta\varepsilon_{\mathrm{rc}}$.
If \(\bar c_j=0\) exactly, the direction has zero objective slope; every feasible step along it remains on the optimal face.
\end{proposition}

\noindent\textit{Proof.}
 Assume first that $x_j$ is at its lower bound and is increased by $\theta \geq 0$. The upper-bound case is obtained by reversing the feasible direction and using the corresponding reduced-cost sign convention. Let $\Delta x_N = \theta e_j$, where $e_j$ is the unit vector for the entering nonbasic coordinate. To preserve the equality constraints, Eq.~\eqref{eq:basic_solution} gives \(\Delta x_B=-\theta B^{-1}A_j\).
The objective displacement is
\begin{align}
    c^T\Delta x
    &=
    c_B^T(-\theta B^{-1}A_j) + c_j\theta \nonumber \\
    &=
    \theta(c_j - c_B^T B^{-1}A_j) \nonumber \\
    &=
    \theta \bar c_j .
    \label{eq:objective_displacement}
\end{align}
The final equality follows from Eq.~\eqref{eq:reduced_cost_single}. Taking absolute values and using $|\bar c_j|\leq \varepsilon_{\mathrm{rc}}$ proves the bound. 
If $x_j$ is instead at its upper bound and is decreased by
$\theta \geq 0$, then $\Delta x_N=-\theta e_j$, so the same calculation gives
$c^T\Delta x=-\theta\bar c_j$. Hence $|c^T\Delta x|\leq\theta\varepsilon_{\mathrm{rc}}$ in either case.
If $\bar c_j=0$, then $c^T\Delta x=0$ for every feasible $\theta$, and the step length is limited only by primal feasibility and bound constraints.

The same calculation extends to simultaneous movements in several nonbasic
variables. If $\Delta x_N=d_N$ is supported on $V_{\mathrm{deg}}$ and the basic
variables are updated by $\Delta x_B=-B^{-1}N d_N$, then
\[
    c^T\Delta x=\bar c_N^T d_N,
    \qquad
    |c^T\Delta x|
    \leq
    \varepsilon_{\mathrm{rc}}\|d_N\|_1.
\]
Thus, the objective displacement is controlled directly by the reduced-cost
tolerance and the magnitude of the feasible displacement.

\subsection{LP-Bound Invariance and Search Non-Invariance}
\label{sec:basis_invariance}

The set \(V_{\mathrm{deg}}\) identifies nonbasic columns with near-zero reduced cost. When such a column admits a feasible simplex step, Proposition~\ref{prop:objective_invariance} shows that the corresponding objective displacement is zero, or bounded by the reduced-cost tolerance. Hence distinct LP states can have the same, or essentially the same, relaxation bound while exposing different solver-facing information.

For mixed-integer search, these differences can matter. Alternative LP states may expose different fractional integer variables, distances from integrality, active rows, reduced costs, and branching candidate sets, and can therefore alter subsequent branching and reoptimization. Section~\ref{sec:selection}
defines the metrics used to characterize these differences.

State selection therefore does not seek a stronger LP bound. It chooses among objective-equivalent or near-equivalent LP states before the subsequent branch-and-bound search. Section~\ref{sec:framework} develops the mechanism used to generate, evaluate, and commit such states.


\section{LP-State Steering Framework}
\label{sec:framework}

The framework uses the reduced-cost degrees of freedom characterized in Section~\ref{sec:geometry} to generate alternative LP states, evaluate them without altering the primary branch-and-bound trajectory, rank the resulting states, and commit a selected state only after compatibility checks.

Figure~\ref{fig:workflow} summarizes the workflow. The framework separates basis steering into four operations. First, candidate generation constructs alternative LP states using reduced-cost and LP-state information from the current degenerate face. Second, isolated scouting evaluates these states without altering the live branch-and bound trajectory. Third, rank-based selection orders candidates using fingerprints of the resulting LP states. Fourth, the intervention policy decides whether a selected state should be committed, skipped, evaluated through bounded lookahead, or executed concurrently as part of a small state portfolio.

\begin{figure}[t]
\centering
\resizebox{\textwidth}{!}{%
\begin{tikzpicture}[
    font=\small,
    node distance=0.8cm and 0.4cm, 
    >=stealth, 
    box/.style={
        draw=blue!60!black,
        fill=blue!4,
        thick,
        rounded corners=3pt,
        align=center,
        text width=2.4cm,       
        minimum height=1.0cm
    },
    outbox/.style={
        draw=teal!60!black,
        fill=teal!4,
        thick,
        rounded corners=3pt,
        align=center,
        text width=2.0cm,       
        minimum height=1.0cm
    },
    arrow/.style={->, >=stealth, thick, draw=black!70}
]

\node[box] (lp) {Current LP state\\(basis, solution, reduced costs)};
\node[box, right=of lp] (gen) {Candidate generation\\(\(K\) perturbation patterns)};
\node[box, right=of gen] (scout) {Isolated scouting\\(candidate LP states)};
\node[box, right=of scout] (rank) {Fingerprinting\\and rank-based selection};
\node[box, right=of rank] (policy) {Intervention policy};

\node[outbox, below=1.2cm of policy, xshift=-3.45cm] (skip) {Skip\\(B)};
\node[outbox, right=0.35cm of skip] (commit) {Commit\\selected state (S)};
\node[outbox, right=0.35cm of commit] (look) {Bounded lookahead\\ (L10)};
\node[outbox, right=0.35cm of look] (race) {Small state\\portfolio (R5)};

\draw[arrow] (lp) -- (gen);
\draw[arrow] (gen) -- (scout);
\draw[arrow] (scout) -- (rank);
\draw[arrow] (rank) -- (policy);

\draw[arrow] (policy.south) -- ++(0,-0.4) -| (skip.north);
\draw[arrow] (policy.south) -- ++(0,-0.4) -| (commit.north);
\draw[arrow] (policy.south) -- ++(0,-0.4) -| (look.north);
\draw[arrow] (policy.south) -- ++(0,-0.4) -| (race.north);

\end{tikzpicture}%
}
\caption{Workflow of the basis-steering framework.}
\label{fig:workflow}
\end{figure}
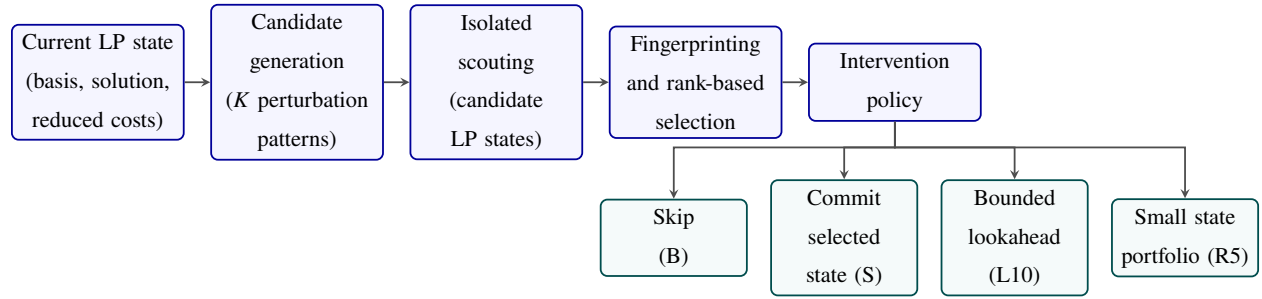

\subsection{Candidate Generation and Isolated Scouting}
\label{sec:candidate_generation}

Candidate generation operates on the reduced-cost degeneracy set \(V_{\mathrm{deg}}\) defined in Eq.~\eqref{eq:vdeg}. Perturbations are supported only on nonbasic variables whose current reduced costs satisfy the numerical tolerance defining \(V_{\mathrm{deg}}\). In the computational implementation, this threshold is configured as \(\varepsilon_{\mathrm{rc}}=10^{-6}\).

For each candidate index \(k=1,\ldots,K\), the generator constructs a deterministic perturbation vector \(d^{(k)}\in\mathbb{R}^n\), with \(d^{(k)}_j=0\) for all \(j\notin V_{\mathrm{deg}}\). The perturbed objective is
\begin{equation}
    c^{(k)} = c + d^{(k)}.
    \label{eq:perturbed_objective}
\end{equation}
For eligible variables, the perturbation direction follows four recurring configurations: two hash-directed patterns, a direction toward the nearest integer value, and a direction toward the opposite integer value. For
fractional variables, the perturbation magnitude is scaled by the distance in the selected direction, subject to a positive lower bound. Variables that are integral in the current relaxation use a hash-directed perturbation. The construction is deterministic and reproducible; the candidate index and variable identifier diversify the resulting perturbation patterns.

The candidate budget \(K\) controls the maximum number of perturbation patterns attempted. Increasing \(K\) provides more opportunities to obtain alternative LP states at the expense of higher scouting cost. The effective candidate pool may contain fewer than \(K\) states because a perturbation can reproduce the baseline state, duplicate a prior candidate, or fail numerical or compatibility checks.

Candidate scouting is performed outside the active branch-and-bound trajectory using SCIP probing. For each candidate, the framework queries the current LP solution and reduced costs, applies the temporary objective perturbation, resolves the modified LP relaxation, and records the resulting LP solution, basis-status information, and structural metrics. The probing changes are reverted before the primary solve resumes. Only a candidate accepted by the intervention policy is subsequently committed for use by branch-and-bound.

Section~\ref{sec:bestofk_audit} evaluates candidate availability through the
best-of-\(K\) audit.

\subsection{Candidate Fingerprints and Rank-Based Selection}
\label{sec:selection}

Candidate generation produces a finite pool of LP states that are equivalent,
or nearly equivalent, with respect to the original LP objective. Selection
therefore ranks candidate states for their downstream use rather than seeking a
stronger relaxation bound.

For each candidate, the framework records a fingerprint of branch-relevant
LP-state quantities. The ten selector features measure fractional repair,
binding structure, reduced-cost degeneracy, branch-candidate exposure, active
row density, fractionality, numerical repair behavior, basis structure, and
displacement from the original LP state. The features are computed after
resolving the candidate LP and therefore describe the state that would be
exposed to subsequent branch-and-bound search.

Let \(f_q(k)\) denote feature \(q\) for candidate \(k\). Each feature is
converted to a normalized within-pool rank \(\rho_q(k)\in[0,1]\), oriented so
that \(0\) corresponds to the preferred direction for that feature and \(1\)
to the least preferred direction. Tied feature values receive their average
rank. The selector score is
\begin{equation}
    R(k)
    =
    \sum_{q=1}^{Q} w_q \rho_q(k),
    \label{eq:rank_selector}
\end{equation}
where \(w_q\geq 0\). Candidates with smaller \(R(k)\) are preferred. Using
within-pool ranks reduces sensitivity to differences in feature scale across
instances and candidate pools.

The selector uses \(Q=10\) rank-normalized features. The feature directions
and weight vector were tuned on the 20-instance pilot set and fixed before
validation. The complete feature definitions and coefficient vector are
reported in Appendix~\ref{app:selector-weight-vector}.

\subsection{Solver Modes and Resource Model}
\label{sec:experimental_modes}

The computational study uses four solver modes. Mode B is the unmodified SCIP reference. Mode S commits the selector's highest-ranked candidate when the intervention checks permit it. Modes L10 and R5 use the selector's top five candidates but differ in how they exploit them: L10 evaluates each candidate for up to 10 branch-and-bound nodes before selecting one to continue, whereas R5 executes the five candidate trajectories concurrently until one proves optimal.

\begin{table}[t]
\TABLE
{Solver Modes and Resource Model\label{tab:modes}}
{\begin{tabular}{@{}l@{\quad}l@{\quad}l@{\quad}l@{\quad}c@{}}
\hline\up
Mode & Policy & Candidate set & Downstream evaluation & Allocated CPUs/task \\ \hline\up
B   & Default SCIP             & None            & None                         & 1 \\
S   & Top-one steering         & Ranked top one  & None                         & 1 \\
L10 & Bounded lookahead        & Ranked top five & Up to 10 nodes per candidate & 6 \\
R5  & Top-five state portfolio & Ranked top five & Run until first proof        & 6\down\\
\hline
\end{tabular}}
{L10 is evaluated only on the 20-instance pilot set. The 130-instance MIPLIB 2017 validation set compares B, S, and R5.}
\end{table}

In R5, the five selected states are launched from the same live root-solve state and continue independently. The first child to prove optimal determines the run outcome; the remaining children are terminated, and the parent does
not reinject the winning state and solve the problem again. B and S use one CPU per task. L10 and R5 use six, corresponding to the parent process and up to five candidate children. Wall-clock comparisons are therefore
 nterpreted under these resource allocations.

\subsection{State Commitment and Intervention Policy}
\label{sec:injection_policy}

The selector ranks candidate states, but ranking does not imply commitment.
Before modifying the live solve, the implementation verifies that the selected
state is available and that its row and column dimensions match the current LP.
If either check fails, the candidate is discarded and SCIP continues from its
current state.

If the checks pass, the saved state is installed in the live LP interface and
the LP is reoptimized from that state by primal simplex. Commitment succeeds
only if state installation succeeds and the resulting LP is optimal. The
SCIP-level LP representation is then synchronized with the LP-interface state
before branch-and-bound continues.

The implementation also supports an optional conservative acceptance gate.
For an LP state \(x\), let \(\phi(x)\) denote the number of fractional binary
variables and let

    $\sigma(x)=\sum_{i\in\mathcal B}
    |x_i-\operatorname{round}(x_i)|$
denote total binary fractionality, where \(\mathcal B\) is the set of binary
variables. When enabled, the gate requires
\begin{equation}
    \phi(x_{\mathrm{base}})-\phi(x_{\mathrm{sel}})
    \ge
    \max\left\{10,\left\lceil0.05\phi(x_{\mathrm{base}})\right\rceil\right\},
    \qquad
    \sigma(x_{\mathrm{sel}})
    \le
    \sigma(x_{\mathrm{base}})+\tau
\end{equation}
where \(\tau\) is the solver feasibility tolerance. This gate is separate from
the learned ranking rule and rejects selected states that do not satisfy the
minimum fractionality-improvement criterion.

After commitment, SCIP resumes its native branching, cutting, and heuristic
machinery. Basis steering changes the LP state presented to these components;
it does not replace the branching rule.

\section{Experimental Design and Metrics}
\label{sec:experiments}

The computational study consists of a controlled witness instance and two
benchmark sets. The witness examines how alternative LP states affect subsequent branch-and-bound search. A 20-instance pilot set is used for candidate-availability analysis,
selector fitting, and policy ablations. A separate 130-instance MIPLIB 2017 validation set is used to evaluate the frozen policies.

\subsection{Experimental Program and Endpoints}
\label{sec:benchmark_protocol}

Comparisons are paired by instance and random seed. For each pair, solver
modes use the same SCIP build, underlying solver configuration, and a two-hour time limit; only mode-specific settings differ. Each solver process is single-threaded. The resource allocations for each mode are given in Section~\ref{sec:experimental_modes}.

The primary endpoint is the proportion of runs that prove optimality within the time limit. Other termination statuses are reported separately.
Secondary endpoints include the Weighted Variability Score (WVS), shifted
geometric means of runtime and recorded node counts, tail behavior, and per-instance differences in the number of optimally solved seeds. Wall-clock comparisons are interpreted under the different resource allocations of the solver modes.

The selector and operational policies are fixed before evaluation on the
130-instance validation set. No selector fitting or policy modification
is performed on the validation data. The experiments use a modified SCIP 9.2.1 build with SoPlex 7.1.3;
hardware, build configuration, and additional reproducibility details
are reported in Appendix~\ref{app:software-hardware-resource-accounting}.

\subsection{Weighted Variability Score}
\label{sec:wvs_metric}

We use the Weighted Variability Score (WVS) as defined in \citet{Adegbindin2026VariabilityScore}. The metric combines normalized runtime, final relative gap, LP iterations, and branch-and-bound nodes into a per-instance measure of seed-to-seed variability. Lower scores indicate greater stability. We retain the normalization, component weights, and treatment of time-limited runs specified in the original definition.

\subsection{Controlled Witness Instance}
\label{sec:toy_instance}

The controlled witness is a 13-variable synthetic instance constructed to
isolate the effect of LP-state variation on subsequent branch-and-bound
behavior. Its root relaxation has a degenerate optimal face on which multiple LP states attain the same relaxation bound while exposing different fractional profiles to the branching rule.

A balance constraint creates the degenerate optimal face, while overlapping cardinality constraints and linking inequalities make different states on that face lead to different downstream search trajectories. The full formulation is given in Appendix~\ref{app:witness-instance}. We evaluate the witness over
1,000 paired seeds.

\subsection{Pilot Set and Best-of-\(K\) Candidate-Availability Audit}
\label{sec:bestofk_protocol}

The pilot set contains 20 benchmark instances selected for variability,
degeneracy, or symmetry. It is used to evaluate candidate availability, fit the rank-based selector, and compare top-one, bounded-lookahead, and top-five portfolio policies before validation.

For each instance and seed, the generator attempts up to \(K=20\)
candidate LP states. Invalid, duplicate, incompatible, or numerically
unusable candidates are discarded. The remaining candidates are
evaluated under the same two-hour downstream time limit.

The primary best-of-\(K\) audit selects one candidate per instance-seed
pair using a lexicographic criterion. Candidates that prove optimality
are preferred to those that do not. Within each optimality-status group,
candidates are ordered by increasing node count, followed by wall-clock
time, LP iterations, and candidate index. Because this selection uses
observed downstream outcomes, best-of-\(K\) is a diagnostic oracle rather
than a deployable solver policy.

\subsection{Validation Protocol and Run Accounting}
\label{sec:validation_protocol}

The validation experiment compares B, S, and R5 on 130 additional
MIPLIB 2017 instances, with a target of 50 paired seeds per instance. L10 is evaluated only on the 20-instance pilot set. 

Run outcomes are parsed from raw solver logs and run-level
records. Each record identifies the instance, seed, mode, termination
status, wall-clock time, node count, LP iterations, primal and dual
bounds, and steering diagnostics.

In R5, the first portfolio child to prove optimality determines the
run outcome. The remaining children are terminated, and the parent
does not rerun the winning candidate. R5 node counts and LP iterations are taken from the winning child's
record; they do not represent the sum of work across all portfolio
children. These measures characterize the winning trajectory rather
than the total computational cost of the portfolio. R5 wall-clock results
are interpreted under its six-CPU allocation.

Replication materials are described in
Appendix~\ref{app:replication-materials}.

\section{Computational Results}
\label{sec:results}
The results proceed from mechanism to held-out validation. We first examine
LP-state causality on the controlled witness, then candidate availability and
policy recovery on the 20-instance pilot set, and finally the frozen policies
on the 130-instance MIPLIB 2017 validation set.

\subsection{Controlled Witness: LP-State Causality}
\label{sec:witness_results}

On the controlled witness, we compare the baseline mode \(B\) with the
LP-state steering mode \(S\) over 1,000 paired random seeds. Both modes solve all 1,000 runs to optimality, so the comparison is not affected by
solved-versus-unsolved censoring. Figure~\ref{fig:witness_distribution} plots the empirical distributions of branch-and-bound node counts and LP iterations. For each computational effort metric \(X\), the figure maps the empirical cumulative distribution function

\[
\label{eq:ecdf}
F_n(x)=\frac{1}{n}\sum_{j=1}^{n}\mathbf{1}\{X_j\leq x\},
\]
where curves lying further to the left correspond to lower computational effort.

\begin{figure}[t]
\centering
\includegraphics[width=0.98\textwidth]{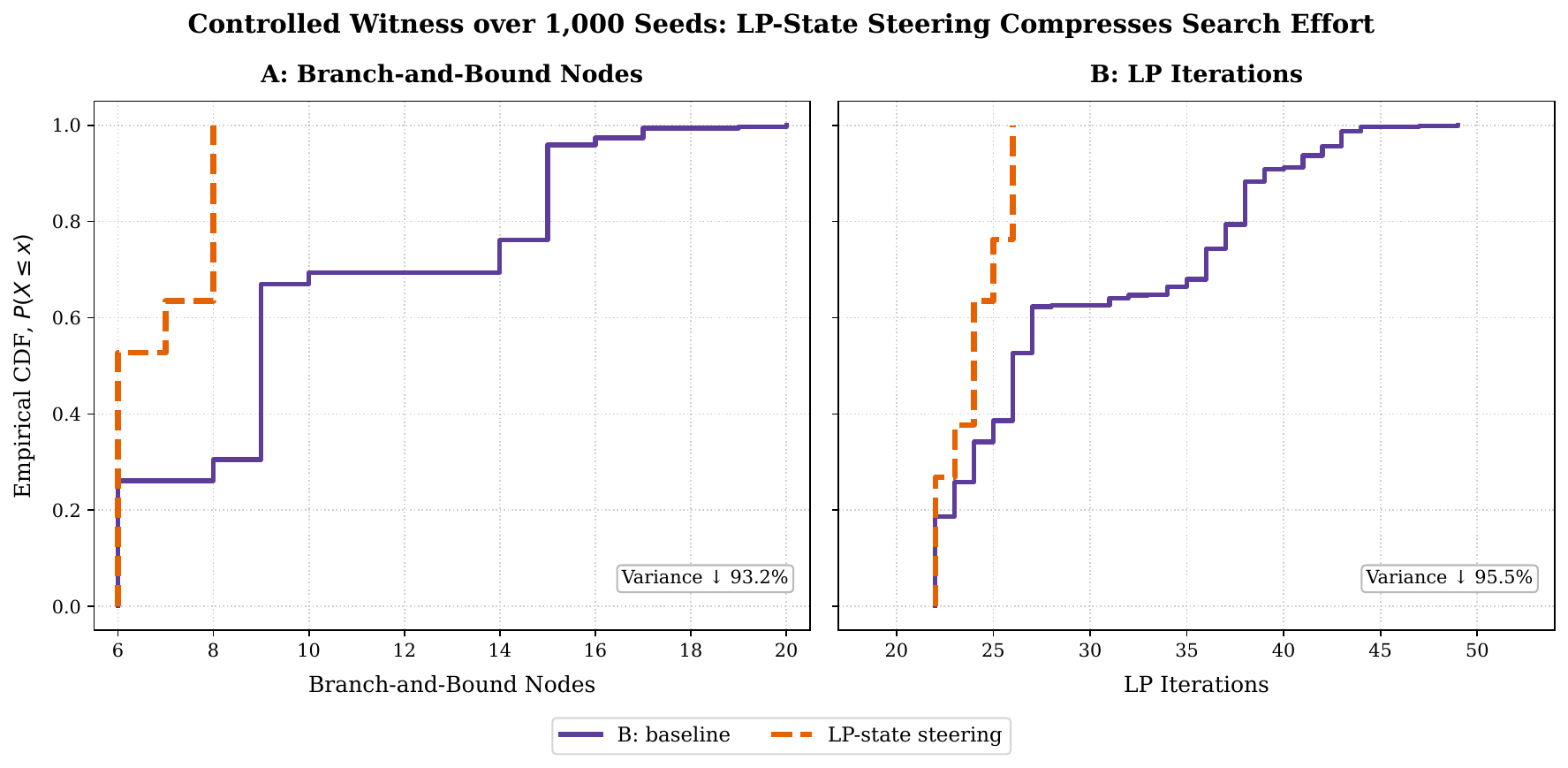}
\caption{Empirical effort distributions over 1000 paired random seeds on the controlled witness instance.}
\label{fig:witness_distribution}
\end{figure}

For branch-and-bound nodes, mode \(B\) has median 9, maximum 20, and standard
deviation 3.58. Mode \(S\) reduces the median to 6, the maximum to 8, and the
standard deviation to 0.93, corresponding to a \(93.2\%\) reduction in raw
node-count variance. For LP iterations, the maximum decreases from 49 to 26
and the raw variance decreases by \(95.5\%\). In paired node-count comparisons,
\(S\) improves 739 runs, ties 261, and exhibits zero performance regressions.

\begin{table}[t]
\TABLE
{Controlled Witness Results\label{tab:witness_controlled}}
{\begin{tabular}{@{}llrrrrrrr@{}}
\hline\up
Metric & Mode & Min & Median & Mean & Max & Std. & CV & MCV \\ \hline\up
Nodes
& \(B\): baseline & 6 & 9 & 10.048 & 20 & 3.578 & 35.6\% & 0.282 \\
& \(S\): LP-state steering & 6 & 6 & 6.838 & 8 & 0.930 & 13.6\% & 0.126 \\ \hline
LP iterations
& \(B\): baseline & 22 & 26 & 29.409 & 49 & 7.062 & 24.0\% & 0.204 \\
& \(S\): LP-state steering & 22 & 24 & 23.956 & 26 & 1.503 & 6.3\% & 0.051\down\\ \hline
\end{tabular}}
{Results are based on 1000 paired random seeds.}
\end{table}

Using the WVS definition from Section~\ref{sec:wvs_metric}, the controlled
witness score decreases from \(0.06899\) under \(B\) to \(0.02276\) under
\(S\), a \(67.0\%\) reduction. Because all trajectories terminate at global optimality and SCIP logs a wall-clock runtime of \(0.00\) seconds for every execution at the reported precision, the runtime and gap metrics are degenerate. Consequently, the witness WVS is entirely driven by the remaining non-degenerate work components: LP iterations and branch-and-bound node counts.

State injection succeeds in all 1,000 steered runs. The median number of
fractional LP branching candidates at the root decreases from 8 before
steering to 1 after the selected LP state is injected. Together, these
diagnostics link the change in the inherited LP state to the change in the
branching-visible fractional profile and the downstream search distribution.

\subsection{Candidate Availability on the 20-Instance Pilot Set}
\label{sec:bestofk_audit}

Prior to selector evaluation, we quantify candidate availability. The best-of-K audit serves as a hindsight diagnostic oracle: for each instance-seed pair, it evaluates the generated LP states and isolates the optimal downstream outcome under the criterion defined in Section~\ref{sec:bestofk_protocol}. This protocol decouples candidate availability from candidate selection. If the generated pool rarely contains a state that improves on the baseline, candidate generation is the limiting step. Conversely, if high-performance alternative states are consistently present, the remaining problem is to identify them using only pre-branching data.

\begin{figure}[t]
\centering
\includegraphics[width=0.96\textwidth]{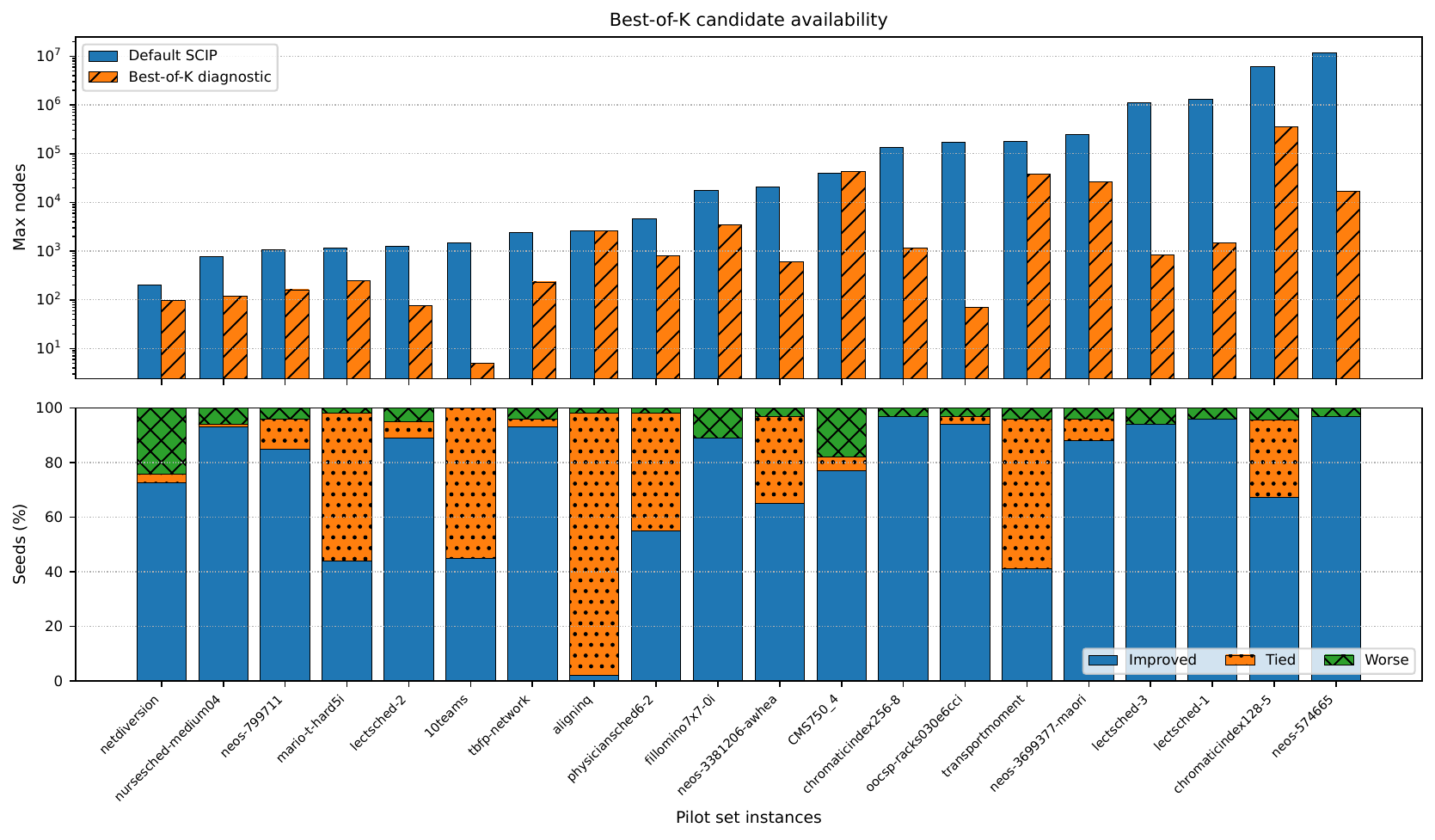}
\caption{Best-of-\(K\) candidate availability on the 20-instance pilot set. The upper panel reports the maximum branch-and-bound node count over the paired seeds on a logarithmic scale; the lower panel reports paired seedwise outcomes relative to default SCIP.}
\label{fig:bestofk_availability}
\end{figure}

\begin{table}[t]
\TABLE
{Best-of-\(K\) Candidate Availability on the 20-Instance Pilot Set\label{tab:bestofk_availability}}
{\begin{tabular}{@{}lrrrrr@{}}
\hline\up
Instance & \(n\) & Default max & Best-of-\(K\) max & Reduction (\%) & I/T/W \\
\hline\up
netdiversion & 99 & 199.0 & 97.0 & 51.3 & 72/3/24 \\
nursesched-medium04 & 100 & 784.0 & 119.0 & 84.8 & 93/1/6 \\
neos-799711 & 100 & 1,088.0 & 161.0 & 85.2 & 85/11/4 \\
mario-t-hard5i & 100 & 1,153.0 & 245.0 & 78.8 & 44/54/2 \\
lectsched-2 & 100 & 1,277.0 & 77.0 & 94.0 & 89/6/5 \\
10teams & 100 & 1,499.0 & 5.0 & 99.7 & 45/55/0 \\
tbfp-network & 100 & 2,447.0 & 232.0 & 90.5 & 93/3/4 \\
aligninq & 100 & 2,616.0 & 2,616.0 & 0.0 & 2/96/2 \\
physiciansched6-2 & 100 & 4,571.0 & 811.0 & 82.3 & 55/43/2 \\
fillomino7x7-0i & 100 & 17,541.0 & 3,415.0 & 80.5 & 89/0/11 \\
neos-3381206-awhea & 100 & 20,494.0 & 613.0 & 97.0 & 65/32/3 \\
CMS750\_4 & 100 & 39,548.0 & 42,670.0 & -7.9 & 77/5/18 \\
chromaticindex256-8 & 100 & 135,071.0 & 1,148.0 & 99.2 & 97/0/3 \\
oocsp-racks030e6cci & 100 & 170,580.0 & 71.0 & 99.9 & 94/3/3 \\
transportmoment & 97 & 176,421.0 & 37,764.0 & 78.6 & 40/53/4 \\
neos-3699377-maori & 100 & 248,617.0 & 26,324.0 & 89.4 & 88/8/4 \\
lectsched-3 & 100 & 1,102,902.0 & 832.0 & 99.9 & 94/0/6 \\
lectsched-1 & 100 & 1,308,191.0 & 1,490.0 & 99.9 & 96/0/4 \\
chromaticindex128-5 & 92 & 6,046,208.0 & 350,394.0 & 94.2 & 62/26/4 \\
neos-574665 & 100 & 11,796,118.0 & 16,905.0 & 99.9 & 97/0/3\down\\
\hline
\end{tabular}}
{Best-of-\(K\) audit metrics on the 20-instance pilot set. Here, \(n\) denotes the number of programmatically valid paired seeds; I/T/W reports the exact frequency of improved, tied, and worse node-count outcomes relative to default SCIP.}
\end{table}

Figure~\ref{fig:bestofk_availability} summarizes candidate availability on the 20-instance pilot set. Along the horizontal axis, instances are ordered from left to right by increasing maximum node count under default SCIP. The upper panel reports the maximum observed branch-and-bound node count over the valid paired seed set. For default SCIP (solid blue bars), this is the largest observed tree size across the sampled seeds. For the best-of-\(K\) diagnostic (hatched orange bars), one candidate is selected for each instance-seed pair using the hindsight rule defined in Section~\ref{sec:bestofk_protocol}, and the panel reports the maximum node count among the resulting selected runs. The logarithmic y-axis accommodates the range from hundreds to millions of nodes.

The lower panel reports how often the best-of-\(K\) alternative outperforms default SCIP over the paired seeds. A paired run is classified as improved (plain blue) or worse (green cross-hatched) first by optimality status; when status does not distinguish the pair, the classification is determined by node count. A tie (orange dotted) denotes an unchanged node count when neither run has a status advantage. The y-axis is normalized to percentages so that the seedwise outcome frequencies are directly comparable across instances.

Taken together, the two panels reveal a large and strongly instance-dependent opportunity for LP-state selection. On several instances with large baseline tails, the best-of-\(K\) diagnostic reduces the maximum observed node count by roughly one to three orders of magnitude. The seedwise panel shows that these tail reductions are often accompanied by frequent improvements across seeds,
while also revealing that improvement frequency and extreme-tail behavior need not coincide.

Table~\ref{tab:bestofk_availability} quantifies the scale of the candidate
availability visible in Figure~\ref{fig:bestofk_availability}. For
\texttt{neos-574665}, the maximum observed node count decreases from
11,796,118 under default SCIP to 16,905 under the best-of-\(K\) diagnostic.
For \texttt{lectsched-1}, it decreases from 1,308,191 to 1,490, and for
\texttt{lectsched-3}, from 1,102,902 to 832. Reductions of 99.9\% also occur
on \texttt{oocsp-racks030e6cci} and \texttt{10teams}, while
\texttt{chromaticindex256-8} achieves a 99.2\% reduction. Across the pilot
set, the best-of-\(K\) diagnostic improves more than half of the paired
outcomes on 16 of the 20 instances.

The table also shows that improvement frequency and tail behavior are distinct.
On \texttt{aligninq}, 96 of 100 paired outcomes tie and the maximum node count
is unchanged. On \texttt{CMS750\_4}, 77 of 100 paired outcomes improve even
though the maximum observed node count increases by 7.9\%. Conversely,
\texttt{10teams} records 45 improvements, 55 ties, and no worse paired
outcomes while reducing the maximum node count from 1,499 to 5.

Figure~\ref{fig:lectsched-bestofk-ecdf} isolates the seed-level behavior for the instance \texttt{lectsched-1} .

\begin{figure}[t]
\centering
\includegraphics[width=0.96\textwidth]{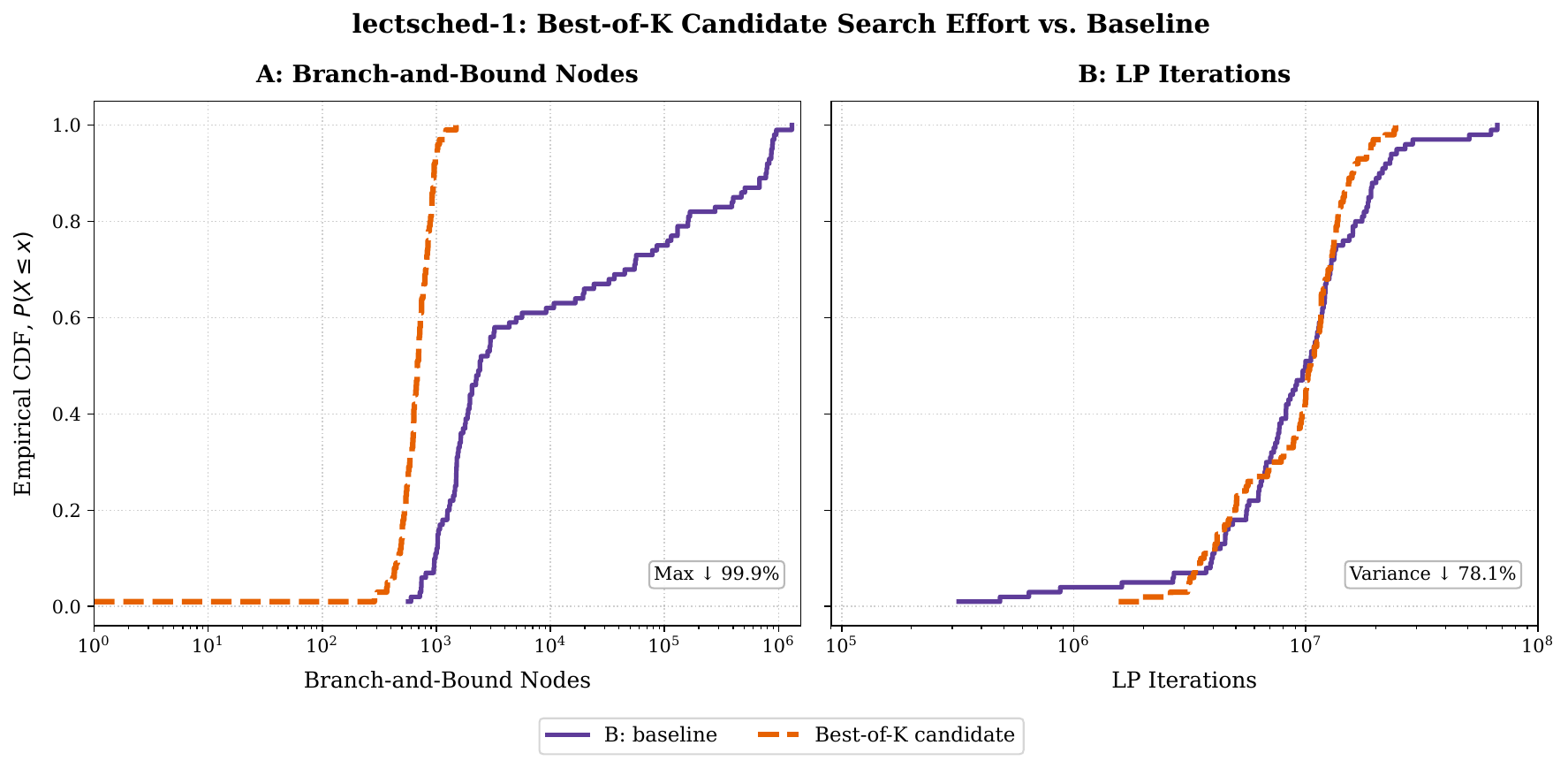}
\caption{Empirical distributions of branch-and-bound nodes and LP iterations
for baseline SCIP and the best generated LP state on \texttt{lectsched-1},
selected retrospectively by downstream node count.}
\label{fig:lectsched-bestofk-ecdf}
\end{figure}

The empirical distributions show that the improvement on
\texttt{lectsched-1} is not confined to an extreme run. In Panel~(A), the
best-of-\(K\) node-count distribution shifts sharply to the left and becomes
substantially more concentrated, while the baseline distribution extends
beyond \(10^6\) nodes. In Panel~(B), the central portions of the LP-iteration
distributions overlap more closely, but the best-of-\(K\) runs have a much
tighter upper tail, corresponding to a \(78.1\%\) reduction in LP-iteration
variance. Thus, the dominant effect on this instance is compression of the
branch-and-bound tree, accompanied by lower dispersion in LP work.

This candidate-availability result defines the target for the implementable
policies: recover as much of the best-of-\(K\) advantage as possible using only
pre-branching LP-state information. Appendix~\ref{app:candidate-abundance}
examines whether useful states occur as isolated candidates or appear more
frequently within the generated pool.

\subsection{Policy Recovery on the 20-Instance Pilot Set}
\label{sec:discovery_policy_results}

The best-of-\(K\) audit establishes that useful LP states are often present in
the generated candidate pool. We next evaluate how much of that opportunity is
recovered by the implementable policies on the same 20-instance pilot set. The
comparison includes the default solver \(B\), top-one steering \(S\), bounded
top-five lookahead \(L10\), and the top-five state portfolio \(R5\).

Figure~\ref{fig:discovery_policy_recovery} isolates \(R5\) against \(B\).
Each row represents one pilot instance, ordered from highest to lowest baseline
WVS. Panel~(A) compares WVS, with \(B\) shown by blue circles and \(R5\) by
orange squares; a leftward shift indicates lower variability. Panel~(B)
compares the maximum observed branch-and-bound node count over the paired seed
set on a logarithmic scale, where a leftward shift indicates a smaller
observed tail. The connecting segments show the direction and magnitude of the
change from \(B\) to \(R5\): longer leftward segments indicate larger
improvements, while an orange marker to the right of its blue counterpart
indicates a regression.

\setlength{\textfloatsep}{12pt plus 0pt minus 0pt}

\begin{figure}[t]
\centering
\includegraphics[width=0.98\textwidth]{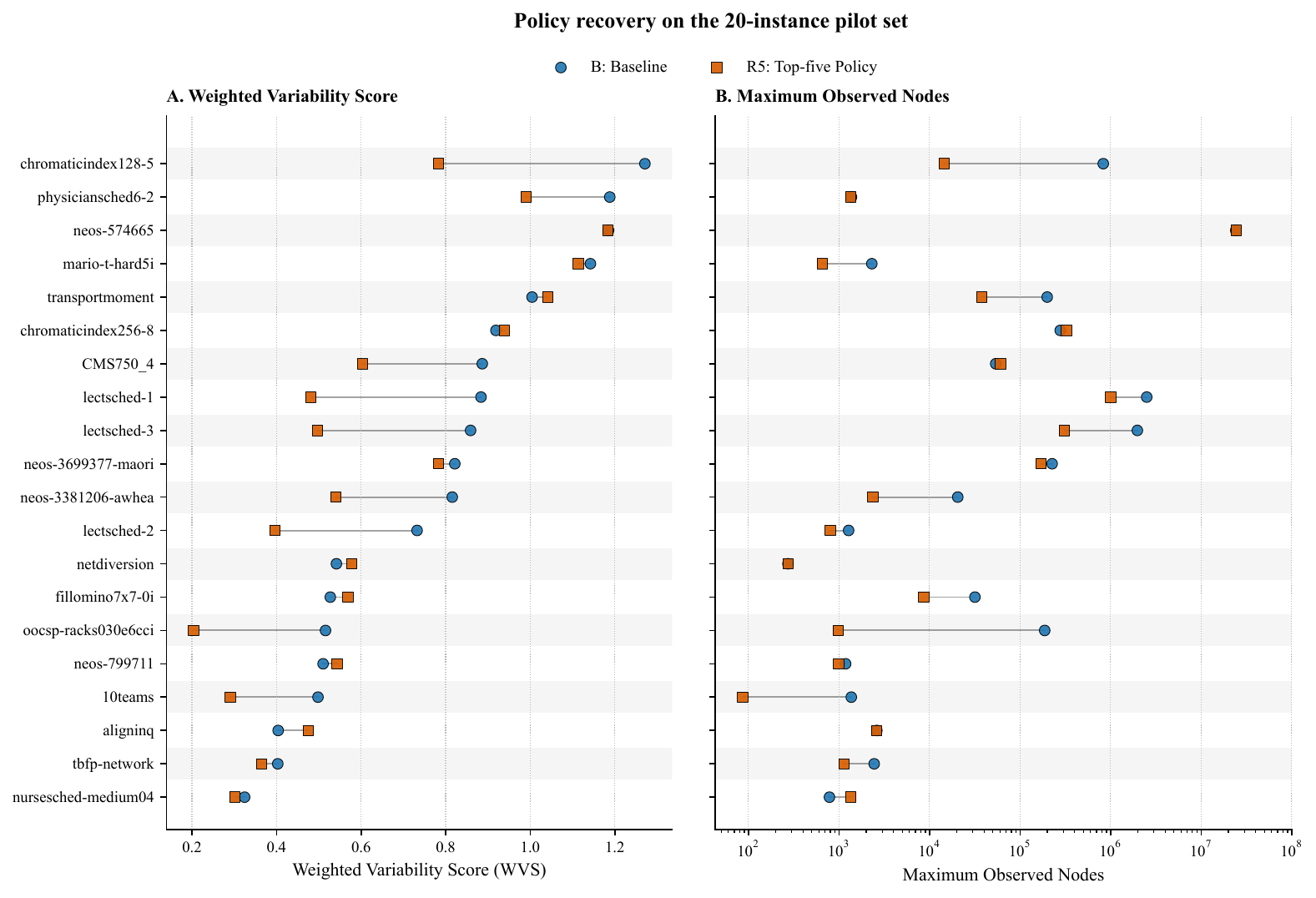}
\caption{Policy recovery on the 20-instance pilot set. Panel~(A) compares WVS
under \(B\) and \(R5\); Panel~(B) compares maximum observed branch-and-bound
nodes over the paired seed set. Instances are ordered by baseline WVS.}
\label{fig:discovery_policy_recovery}
\end{figure}

The two panels show that WVS and the observed node-count tail capture distinct
aspects of policy performance. Several instances move substantially to the
left in both panels, while others improve on only one measure. For example,
on \texttt{transportmoment}, \(R5\) reduces the maximum node count from
199,204 to 37,764 while WVS changes from 1.004 to 1.041.

\begin{table}[t]
\TABLE
{Policy Recovery on the 20-Instance Pilot Set
\label{tab:discovery_policy_recovery}}
{\begin{tabular}{@{}lrrrrrrrrr@{}}
\hline\up
Instance & \(n\) & B max & S max & L10 max & R5 max & WVS B & WVS S & WVS L10 & WVS R5 \\
\hline\up
netdiversion & 100 & 273 & 273 & 273 & 273 & 0.542 & 0.540 & 0.573 & 0.578 \\
nursesched-medium04 & 100 & 784 & 917 & 1,344 & 1,348 & 0.324 & 0.294 & 0.290 & 0.302 \\
neos-799711 & 100 & 1,183 & 1,183 & 1,058 & 994 & 0.510 & 0.502 & 0.537 & 0.544 \\
lectsched-2 & 100 & 1,277 & 1,277 & 1,498 & 799 & 0.732 & 0.566 & 0.574 & 0.396 \\
10teams & 100 & 1,367 & 985 & 1,454 & 86 & 0.498 & 0.455 & 0.407 & 0.291 \\
physiciansched6-2 & 100 & 1,372 & 2,143 & 1,160 & 1,352 & 1.188 & 1.068 & 1.100 & 0.990 \\
mario-t-hard5i & 100 & 2,304 & 2,047 & 1,380 & 654 & 1.142 & 1.185 & 1.192 & 1.114 \\
tbfp-network & 100 & 2,447 & 3,320 & 2,978 & 1,133 & 0.403 & 0.511 & 0.406 & 0.364 \\
aligninq & 100 & 2,616 & 2,616 & 2,616 & 2,616 & 0.404 & 0.401 & 0.381 & 0.476 \\
neos-3381206-awhea & 100 & 20,494 & 8,404 & 6,397 & 2,368 & 0.816 & 0.825 & 0.759 & 0.540 \\
fillomino7x7-0i & 100 & 31,717 & 41,996 & 44,343 & 8,649 & 0.527 & 0.540 & 0.640 & 0.569 \\
CMS750\_4 & 100 & 53,817 & 55,383 & 67,197 & 60,683 & 0.887 & 0.894 & 0.647 & 0.603 \\
oocsp-racks030e6cci & 100 & 187,000 & 349,940 & 510,245 & 979 & 0.516 & 0.671 & 0.726 & 0.203 \\
transportmoment & 100 & 199,204 & 199,204 & 64,878 & 37,764 & 1.004 & 1.052 & 1.005 & 1.041 \\
neos-3699377-maori & 100 & 225,948 & 159,972 & 170,885 & 170,885 & 0.822 & 0.740 & 0.727 & 0.783 \\
chromaticindex256-8 & 100 & 277,646 & 283,080 & 365,109 & 326,096 & 0.919 & 0.911 & 0.935 & 0.939 \\
chromaticindex128-5 & 100 & 829,345 & 756,177 & 588,760 & 14,524 & 1.271 & 1.176 & 1.151 & 0.783 \\
lectsched-3 & 100 & 1,973,974 & 2,191,729 & 2,546,842 & 308,961 & 0.859 & 0.798 & 0.824 & 0.497 \\
lectsched-1 & 100 & 2,504,428 & 2,489,672 & 2,731,104 & 1,006,427 & 0.884 & 0.856 & 0.885 & 0.481 \\
neos-574665 & 100 & 24,128,761 & 24,095,300 & 29,090,895 & 24,413,860 & 1.185 & 1.183 & 1.176 & 1.184\down\\
\hline
\end{tabular}}
{All comparisons use common paired instance-seed runs. B and S use one CPU
per task. L10 and R5 use six CPUs per task. Node columns report the
branch-and-bound node count of the returned solve under each mode's resource
model.}
\end{table}

Table~\ref{tab:discovery_policy_recovery} reports the full four-mode
comparison. On \texttt{oocsp-racks030e6cci}, \(R5\) reduces the maximum
observed node count from 187,000 to 979 while WVS decreases from 0.516 to
0.203. On \texttt{chromaticindex128-5}, the maximum decreases from 829,345
to 14,524 and WVS from 1.271 to 0.783. On \texttt{lectsched-1}, the maximum
decreases from 2,504,428 to 1,006,427 and WVS from 0.884 to 0.481; on
\texttt{lectsched-3}, the corresponding changes are 1,973,974 to 308,961
nodes and 0.859 to 0.497 in WVS. On \texttt{10teams}, the maximum node count
decreases from 1,367 to 86 and WVS from 0.498 to 0.291.

The gains are not uniform. On \texttt{neos-574665}, \(R5\) has essentially
the same WVS as \(B\) and a slightly larger maximum node count
(24,413,860 versus 24,128,761). On \texttt{aligninq}, the maximum node count
is unchanged at 2,616 while WVS increases from 0.404 to 0.476. These cases,
together with \texttt{transportmoment}, show that candidate recovery can affect
tail behavior and overall variability differently across instances.

The four-mode comparison also shows the value of retaining several highly
ranked candidates. Modes \(S\) and \(L10\) recover useful states on some
instances, but several of the largest pilot-set gains occur under \(R5\),
which retains the selector's top five states rather than committing to a
single ranked candidate.

Additional pilot-set profiles for median LP iterations and median wall-clock time are reported in Appendix~\ref{app:pilot-component-profiles}.

These pilot results lead to the held-out validation in
Section~\ref{sec:benchmark_results}. \(L10\) remains a pilot ablation, while \(B\), \(S\), and \(R5\) proceed to the 130-instance validation experiment.

\subsection{Validation on 130 MIPLIB 2017 Instances}
\label{sec:benchmark_results}
The final experiment evaluates the frozen selector and fixed policies on a
held-out set of 130 additional MIPLIB 2017 instances, with a target of 50
paired seeds per instance. The comparison includes the baseline solver \(B\),
top-one steering \(S\), and the top-five portfolio policy \(R5\). The validation set was constructed to cover both instance families and baseline variability roles. It includes graph-symmetry, industrial/numerical, market and machine-learning, network, NEOS, packing/production, scheduling, and set-assignment families, together with low-, medium-, and high-variability
instances. The family composition of the validation set and its distribution across
baseline tractability regimes are reported in
Appendix~\ref{app:validation-family-regimes}.

\begin{table}[t]
\TABLE
{Frozen Validation Summary on 130 Held-Out MIPLIB 2017 Instances\label{tab:validation_130_summary}}
{\begin{tabular}{@{}llrrrrrrrr@{}}
\hline\up
Mode  & Runs & Opt. & TL & Other & Solve rate & \(\Delta\) opt. & SGM time & SGM nodes & Median WVS \\
\hline\up
\(B\)  & 6,492 & 3,792 & 2,699 & 1 & 58.4\% & - & 893.4 & 2,574.7 & 0.213 \\
\(S\)  & 6,473 & 3,812 & 2,661 & 0 & 58.9\% & +20 & 899.8 & 2,602.0 & 0.215 \\
\(R5\)  & 6,498 & 3,894 & 2,604 & 0 & 59.9\% & +102 & 823.1 & 1,198.0 & 0.210\down\\
\hline
\end{tabular}}
{Aggregate validation results over all recorded runs. TL denotes time-limit termination, and \(\Delta\) opt.\ reports the difference in total optimal completions relative to \(B\). B and S use one CPU per task; R5 uses six. SGM time and SGM nodes use shifts of 10 seconds and 100 nodes, respectively.}
\end{table}

Table~\ref{tab:validation_130_summary} reports the aggregate validation
results. The headline quantities are solve completion, shifted geometric means
of wall-clock time and recorded node count, and median WVS. Relative to the
observed baseline total, \(R5\) increases optimal completions from 3,792 to
3,894, raises the solve rate from 58.4\% to 59.9\%, reduces the shifted
geometric mean of wall-clock time from 893.4 to 823.1 seconds, and reduces the
shifted geometric mean of recorded node count from 2,574.7 to 1,198.0.
Mode \(S\) records 20 additional optimal completions but does not improve the
shifted geometric means of time or nodes. Because the number of recorded runs
differs slightly across modes, the \(+20\) and \(+102\) entries are differences
in observed optimal-run totals rather than paired seedwise differences.

The WVS effect across all 130 instances is regime-dependent rather than
uniform. The median \(R5/B\) WVS ratio is 0.969; \(R5\) improves WVS on 73
instances and worsens it on 57. A paired Wilcoxon signed-rank test across the
130 instance-level WVS values gives \(p=0.1686\), so the full validation set
does not support a uniform WVS reduction. This motivates the regime-level
analysis below.

\begin{figure}[t]
\centering
\includegraphics[width=0.98\textwidth]{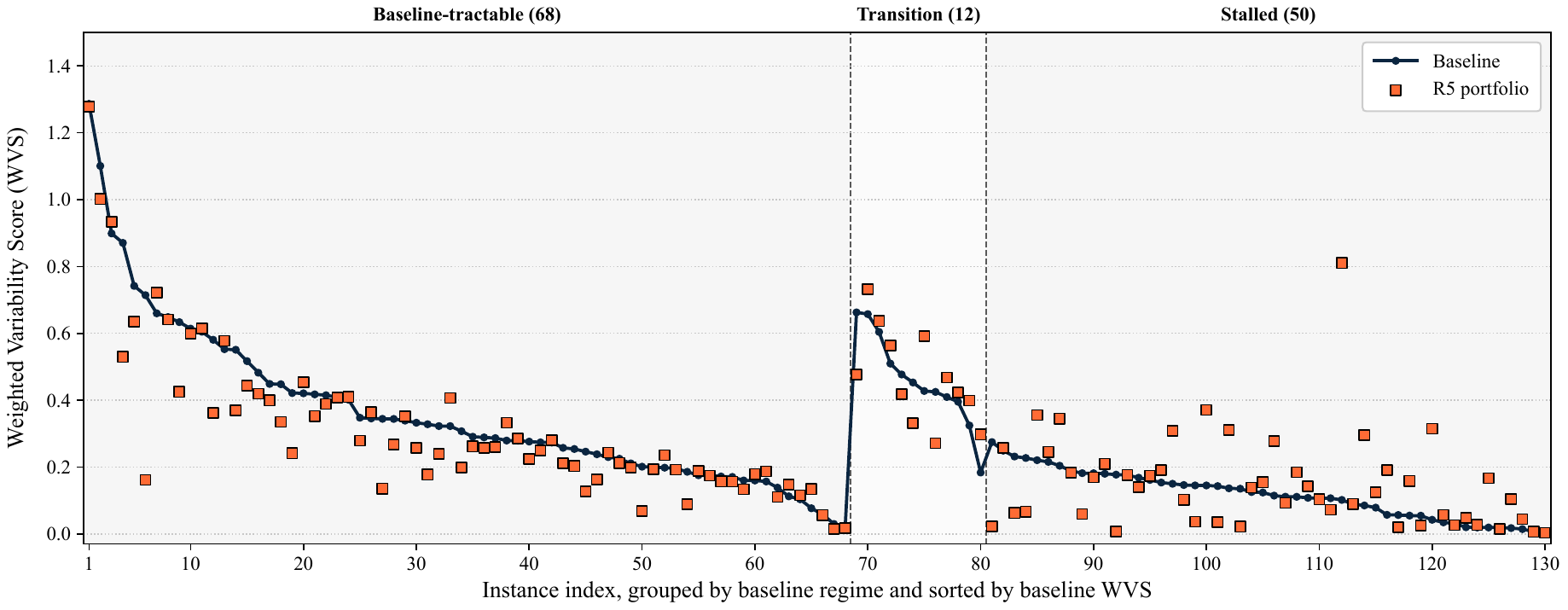}
\caption{Instance-level WVS on the 130-instance held-out validation set.
The blue line reports baseline WVS (\(B\)); the orange squares report WVS
under \(R5\). Instances are grouped by baseline tractability regime and sorted
within each regime by baseline WVS.}
\label{fig:validation_wvs_curve}
\end{figure}

Figure~\ref{fig:validation_wvs_curve} shows the instance-level WVS behavior
behind the aggregate result. To separate variability behavior from increasing
time-limit saturation, the validation instances are partitioned by observed
baseline solve rate. Instances with baseline solve rate at least \(0.90\) are
classified as baseline-tractable; those with solve rate strictly between
\(0.10\) and \(0.90\) form the transition regime; and those with solve rate at
most \(0.10\) are classified as stalled. Within each regime, the blue line gives the baseline WVS profile and the orange
squares give the corresponding \(R5\) values. In the baseline-tractable regime,
the orange markers lie below the blue curve for most instances. The transition
regime contains both WVS increases and decreases. In the stalled regime, the responses are more dispersed as the
two-hour time limit dominates a larger fraction of runs.

\begin{table}[t]
\TABLE
{Validation Behavior by Baseline Tractability Regime\label{tab:validation_regimes}}
{\begin{tabular}{@{}lrrrrrcrrr@{}}
\hline\up
Regime & Inst. & B solve & R5 solve & B WVS & R5 WVS & R5 WVS I/T/W & Avg. gain & Avg. loss & \(p\)-value \\
\hline\up
Baseline-tractable & 68 & 0.996 & 0.998 & 0.364 & 0.315 & 48/0/20 & 0.081 & 0.029 & \(7.90\times 10^{-5}\) \\
Transition & 12 & 0.655 & 0.738 & 0.461 & 0.467 & 4/0/8 & 0.130 & 0.075 & 0.850 \\
Stalled & 50 & 0.005 & 0.024 & 0.118 & 0.151 & 21/0/29 & 0.068 & 0.106 & 0.121 \down\\
\hline
\end{tabular}
}
{Regimes are defined by baseline solve rate: baseline-tractable if B solves at least 90\% of seeds, transition if B solves between 10\% and 90\% of seeds, and stalled if B solves at most 10\% of seeds. Average gain is the mean value of \( \mathrm{WVS}_B-\mathrm{WVS}_{R5}\) over instances where \(R5\) improves WVS. Average loss is the mean value of \( \mathrm{WVS}_{R5}-\mathrm{WVS}_B\) over instances where \(R5\) worsens WVS.}
\end{table}

Table~\ref{tab:validation_regimes} quantifies these regime differences. On the
68 baseline-tractable instances, \(R5\) improves WVS on 48 and worsens it on
20. The median per-instance \(R5/B\) WVS ratio is 0.910, corresponding to a
9\% reduction, and a paired Wilcoxon signed-rank test gives
\(p=7.90\times10^{-5}\). Among the improving instances, the mean WVS decrease
is 0.081; among the worsening instances, the mean increase is 0.029. The
largest observed decrease is 0.553 and the largest increase is 0.084.
This is the strongest WVS result in the validation experiment.

The 12 transition instances show a different effect. The average solve rate
increases from 0.655 under \(B\) to 0.738 under \(R5\), while mean WVS changes
from 0.461 to 0.467. The principal result in this regime is therefore increased
solve completion rather than lower WVS.

The 50 stalled instances are strongly affected by the time limit. Their average
baseline solve rate is 0.005, increasing to 0.024 under \(R5\), while mean WVS
increases from 0.118 to 0.151. The gap-variability component decreases
descriptively from 0.6025 under \(B\) to 0.4859 under \(R5\), but this regime
does not support a uniform WVS-reduction claim.

Several individual instances show how these aggregate gains arise. On
\texttt{glass4}, \(R5\) reduces the maximum observed node count from
3,884,554 to 492,775, lowers WVS from 0.581 to 0.362, and increases solved
seeds from 49 to 50. On \texttt{neos-2624317-amur}, the maximum observed node
count decreases from 2,233,413 to 150,350, WVS decreases from 1.101 to 1.002,
and solved seeds increase from 46 to 49. On \texttt{supportcase34}, the maximum
observed node count decreases from 31,370 to 10,144, WVS from 0.663 to 0.477,
and solved seeds increase from 34 to 50.

The validation set also contains instances where the observed node-count tail
increases under \(R5\). The maximum observed node count rises from 41,428,151
to 48,248,381 on \texttt{gen-ip016}, from 26,829,014 to 28,380,088 on
\texttt{pigeon-13}. Additional stalled cases in which recorded node counts
decrease without a corresponding solve-completion gain are reported in
Appendix~\ref{app:validation_diagnostics}.

Overall, the held-out validation shows a regim e-dependent effect: \(R5\)
reduces WVS on baseline-tractable instances, improves solve completion in the
transition regime, and does not produce a uniform WVS reduction on stalled
instances.

\subsection{Summary of Empirical Claims}
\label{sec:empirical_claim_summary}

The empirical evidence supports three central claims. First, the controlled witness
isolates the mechanism: changing the LP state inherited by branch-and-bound
changes the fractional profile exposed to the branching rule and can change
the downstream node-count distribution.

Second, the 20-instance pilot set establishes candidate availability and the
extent to which implementable policies recover that opportunity. The
best-of-\(K\) audit shows that useful LP states are often present among the
generated candidates, while \(S\), \(L10\), and \(R5\) measure how much of
that opportunity can be recovered without access to downstream outcomes.
The diagnostic oracle serves as a retrospective availability diagnostic,
not a deployable solver policy.

Third, the held-out validation shows that the frozen policy has
regime-dependent effects. The portfolio policy \(R5\) improves aggregate solve
completion and reduces the shifted geometric means of wall-clock time and
recorded node count across the full validation set. It significantly reduces
WVS on baseline-tractable instances and increases solve completion in the
transition regime, while stalled instances do not show a uniform WVS reduction
and are strongly affected by time-limit saturation.


\section{Discussion}
\label{sec:discussion}

\subsection{What the Experiments Show}
\label{sec:discussion_results}

The experiments show that LP-equivalent states need not be search-equivalent. In degenerate MIP relaxations, alternative optimal LP states can preserve the relaxation bound exactly while exposing different fractional structures, reduced-cost profiles, and branching candidates. The state passed from continuous optimization to branch-and-bound can therefore affect the subsequent integer search even when the relaxation bound is unchanged.

Taken together, the three experiments separate mechanism, opportunity, and deployment. The witness identifies a direct path from the inherited LP state to downstream search behavior; the pilot shows that alternative states with better search outcomes can exist within the generated candidate set; and the held-out validation tests whether a frozen policy can exploit that opportunity without using downstream information. The combined evidence therefore points to LP-state steering as a conditional solver intervention: its value depends on whether alternative states differ in search-relevant structure and whether the policy can identify and exploit those differences under the available resource budget.

\subsection{Consequences for Solver Design}
\label{sec:discussion_lessons}

LP optimality and downstream search quality are different objectives. An LP
solution can be fully satisfactory as a certificate of the relaxation bound
while providing a poor state from which to continue the MIP search. Within
branch-and-bound, the returned LP state also determines the fractional
solution, reduced-cost information, branch-candidate set, and starting point
for subsequent LP reoptimization. Treating LP-equivalent states as
interchangeable can therefore discard state information that matters to the
integer search.

Candidate generation, candidate selection, and state commitment are distinct
algorithmic problems. Generation determines which alternative LP states are
exposed. Selection ranks those states using information available before
downstream search outcomes are known. Commitment determines whether a selected
state is compatible with the live LP and should replace the current state.
Separating these stages makes failure modes identifiable: useful states may
never be generated, they may be ranked poorly, or a selected state may be
unsuitable for commitment.

The results also favor selective rather than ubiquitous intervention.
Degeneracy creates the possibility of alternative states, but it does not
imply that steering should occur at every node or on every instance. Candidate
generation and evaluation incur additional LP reoptimization and selection
overhead, and an unnecessary intervention can disturb a trajectory that was
already favorable. A practical implementation should therefore steer only
when the available state information justifies the additional computation.
LP-state steering complements branching, cutting, presolve, and primal
heuristics rather than replacing them.

Solver integration imposes an additional state-consistency requirement.
Selecting a state through the LP interface is algorithmically effective only
if the solver components responsible for branching subsequently observe the
same committed state. A mismatch between the installed LP state and the state
seen by branching can eliminate the intended intervention. We do not replace
the solver's native branching rule; we control the LP state from which that
rule operates.

More broadly, degeneracy in a MIP relaxation is not only a numerical
phenomenon. When alternative states on the same optimal or near-optimal face
lead to different downstream search behavior, degeneracy creates an
algorithmic choice over the state inherited by branch-and-bound.

\subsection{Resource Accounting and Threats to Validity}
\label{sec:discussion_validity}

The most immediate qualification is the resource asymmetry of the top-five
portfolio. Modes \(B\) and \(S\) use one CPU per task, whereas \(R5\) uses six.
Accordingly, \(R5\) should not be interpreted as a one-core replacement for
the baseline solver. It evaluates whether several selected LP states can be
exploited concurrently under an explicit larger resource budget. The resource
model must therefore be considered together with the wall-clock,
solve-completion, node-count, and variability results.

Node and LP-iteration accounting also differs under portfolio execution. For
\(R5\), the reported node and LP-iteration metrics characterize the winning
child rather than aggregate work across all concurrent portfolio members.
These quantities describe the trajectory that proves optimality, not the total
computational work consumed by the portfolio. This distinction is necessary
when comparing search trajectories across the different resource models.

Selector calibration is a second source of uncertainty. The rank-based
selector and policy choices were developed on the 20-instance pilot set and
frozen before the 130-instance validation experiment. This separation prevents
the held-out instances from entering the tuning loop, but it does not imply
that the selected features, weights, or candidate budget are uniquely optimal.
The validation results therefore support the tested frozen policy rather than
all possible LP-state selectors.

LP-state validity also depends on the formulation against which the state was
generated. Added cuts can change the LP row space, while bound changes and
other node-level modifications can make a saved state stale or incompatible
with the current relaxation. The implementation therefore checks compatibility
before commitment and reoptimizes an installed state before native search
continues. A tighter integration would perform state selection after the
relevant separation and LP-modification steps for the current relaxation and
immediately before branching consumes the resulting LP state.

Finally, the intervention is inherently instance-dependent. Some relaxations
contain little useful reduced-cost degeneracy, and some generated states differ
algebraically without materially changing the branch-relevant fractional
structure. In other cases, the cost of generating or evaluating alternatives
may exceed their downstream benefit. The neutral and adverse validation cases
are therefore part of the empirical boundary of the method rather than
exceptions to be removed from the analysis.

\subsection{Limitations and Future Directions}
\label{sec:discussion_limits}

The current implementation incurs computational overhead from candidate
generation, LP reoptimization, compatibility checking, state commitment, and,
for portfolio execution, concurrent solver processes. Those costs are
justified only when they are offset by improvements in solve completion,
search effort, runtime, or variability. A production implementation should
therefore make steering cost-aware rather than applying a fixed intervention
policy indiscriminately.

Candidate generation and selection remain natural targets for improvement.
Generation could incorporate additional information about persistent branching
candidates, basis quality, cut activity, or other solver state when constructing
alternative LP states. Selection could likewise adapt to instance and search
context while retaining the requirement that decisions use only information
available before downstream outcomes are observed. The objective is not simply
to generate more candidates or improve ranking accuracy in isolation, but to
increase the probability of exposing and selecting useful states at a
computational cost justified by the downstream search.

The portfolio policy can also be made resource-adaptive. The present \(R5\)
configuration evaluates a fixed top-five state portfolio under a six-CPU
allocation. A production policy could decide when concurrent execution is
justified, how many candidate states to retain, and when to terminate
unpromising trajectories. Such decisions should account explicitly for CPU
consumption as well as wall-clock time, solve completion, search effort, and
tail behavior.

Closer solver integration offers a further direction. Solver interfaces could
provide more direct support for enumerating LP-equivalent states, preserving a
selected state through the transition to branching, and recording the state
actually observed by downstream solver components. Degeneracy-aware simplex
tie-breaking, branch-candidate-aware state selection, and state portfolios
constructed immediately before branching are natural extensions of the
framework.

The broader opportunity is to treat LP-state inheritance as an explicit part
of MIP algorithm design. The experiments show that the state returned by the
relaxation can contain search-relevant information beyond its objective value.
Better generation, selection, commitment, and resource-allocation policies can
therefore exploit a solver degree of freedom that conventional
branch-and-bound implementations largely leave implicit.

\section{Conclusion}
\label{sec:conclusion}

This paper studies MIP performance variability through LP-state inheritance. In degenerate relaxations, the LP objective does not uniquely determine the state passed to branch-and-bound. Alternative optimal LP states can preserve
the relaxation bound exactly while exposing different fractional structures, reduced-cost profiles, and branching candidates. Near-optimal states permit small objective changes controlled by the reduced-cost tolerance and perturbation size. The inherited LP state is therefore an algorithmically relevant part of the transition from continuous relaxation to integer search.

We introduced an LP-state steering framework that makes this choice explicit.
The framework generates alternative states on or near degenerate LP faces, evaluates them without altering the live solve, ranks them using information available before downstream search, and either commits a selected state or
executes a small state portfolio under an explicit resource model. The solver's native branching rule is unchanged; the intervention controls the LP state from which that rule operates.

The computational study provides evidence at three levels. The controlled witness isolates the mechanism by showing that a change in the inherited LP state can alter the fractional profile and downstream node-count distribution.
The 20-instance pilot set shows that useful alternatives are often present among generated candidates and that implementable policies can recover part of this opportunity without observing downstream outcomes. On the separate
130-instance MIPLIB 2017 validation set, the frozen \(R5\) policy increases recorded optimal completions from 3,792 to 3,894 and reduces the shifted geometric mean of recorded node count from 2,574.7 to 1,198.0. Its WVS effect is regime-dependent: the median per-instance \(R5/B\) WVS ratio is 0.910 on baseline-tractable instances, corresponding to a 9\% reduction, while transition and stalled instances exhibit different behavior.

These results identify LP-state inheritance as a concrete solver-design choice at the interface between continuous relaxation and combinatorial search. LP-state steering does not replace branching, cutting, presolve, or primal heuristics; it controls the relaxation state from which downstream search proceeds. The broader implication is that LP-equivalent states need not be search-equivalent, and that deliberate selection among them can reduce search effort and performance variability.
\clearpage

%
%
%


\begin{APPENDICES}


\section{Controlled Witness Specification}
\label{app:witness-instance}

This appendix records the algebraic specification of the controlled witness
used in Sections~\ref{sec:toy_instance} and~\ref{sec:witness_results}.
Its root LP relaxation contains a degenerate optimal face: multiple LP states attain the same root bound while exposing different fractional branching profiles to the integer search.

\subsection{Full Formulation}
\label{app:witness-formulation}

The witness instance has twelve binary variables and one continuous slack variable:
\begin{align}
\max \quad
& 10 \sum_{j=1}^{12} x_j \label{eq:app-witness-objective} \\
\text{s.t.} \quad
& \sum_{j=1}^{12} x_j + y = 6.5, \label{eq:app-witness-balance} \\
& \sum_{j=1}^{6} x_j \leq 5, \label{eq:app-witness-first-block} \\
& \sum_{j=7}^{9} x_j \leq 1.5,
\quad
\sum_{j=8}^{10} x_j \leq 1.5, \nonumber \\
& \sum_{j=9}^{11} x_j \leq 1.5,
\quad
\sum_{j=10}^{12} x_j \leq 1.5, \label{eq:app-witness-overlapping-blocks} \\
& x_7 + x_{12} \leq 1.1,
\quad
x_1 + x_7 \leq 1.2, \nonumber \\
& x_2 + x_8 \leq 1.2,
\quad
x_3 + x_9 \leq 1.2, \label{eq:app-witness-linking-constraints} \\
& 0 \leq y \leq 1,
\qquad
x_j \in \{0,1\},
\quad j=1,\dots,12 . \label{eq:app-witness-variable-bounds}
\end{align}

In the continuous relaxation, the balance constraint implies
\(\sum_j x_j=6.5-y\). Hence any feasible state with \(y=0\) satisfies
\(\sum_j x_j=6.5\) and reaches the LP objective value \(65\).
The overlapping cardinality and linking inequalities intersect this optimal
face in different ways, allowing LP states with the same bound to expose
different fractional candidates to the branching rule.

\subsection{State-Injection Audit}
\label{app:witness-state-audit}

The steered runs include an explicit state-reproduction audit.
Across the 1000 paired seeds, the root hook executes in all 1000 steered runs, the selected LP state is injected in all 1000 runs, and no injection skips are recorded. Before steering, the median number of fractional LP branching candidates at the root is 8. After injection, the median is 1 and the maximum is 2. At the branch-and-bound level, the steered mode improves 739 paired seeds, ties 261, and is never worse in node count.

This audit verifies that the intended intervention reaches the branching
interface: the root LP bound remains unchanged while the fractional profile
passed to branch-and-bound changes before the search tree is explored.

\newpage

\section{Frozen Selector Specification}
\label{app:selector-specification}

This appendix records the feature directions and weights frozen after the
20-instance pilot study and used without further tuning in the held-out
validation experiment.

\subsection{Frozen Selector Weight Vector}
\label{app:selector-weight-vector}

\begin{table}[htbp]
\TABLE
{Frozen Selector Weight Vector\label{tab:selector-weight-vector-table}}
{\footnotesize
\begin{tabular}{llr}
\hline
Feature & Direction & Weight \\
\hline
LF\_NFIX              & HI & 0.2735514148 \\
LF\_BIND              & LO & 0.1781886547 \\
LF\_RC\_ZERO\_RATIO   & LO & 0.1664482409 \\
LF\_BRANCH\_NLPCANDS  & LO & 0.1048658177 \\
LF\_RC\_ABS\_MINPOS   & HI & 0.0671931565 \\
LF\_ACTIVE\_DENSITY   & LO & 0.0518708859 \\
LF\_FRAC\_VAR         & LO & 0.0515237686 \\
LF\_REPAIR\_NCLOSE0   & LO & 0.0476371452 \\
LF\_BASIC\_COUNT      & HI & 0.0321810832 \\
LF\_FACE\_L2          & HI & 0.0265398326 \\
\hline
\end{tabular}}
{HI indicates that larger feature values are preferred; LO indicates that
smaller values are preferred. The reported weights sum to one.}
\end{table}

\subsection{Feature Interpretation}
\label{app:selector-feature-definitions}

\begin{table}[htbp]
\TABLE
{Selector Feature Interpretation\label{tab:selector-feature-definitions-table}}
{\footnotesize
\begin{tabular}{lp{0.68\textwidth}}
\hline
Feature & Interpretation \\
\hline
LF\_NFIX
& Repair signal for variables that are fractional in the base LP state. \\

LF\_BIND
& Binding-row signal in the candidate LP state. \\

LF\_RC\_ZERO\_RATIO
& Ratio of zero or near-zero reduced-cost variables. \\

LF\_BRANCH\_NLPCANDS
& Number of LP-visible branching candidates. \\

LF\_RC\_ABS\_MINPOS
& Minimum positive absolute reduced-cost signal. \\

LF\_ACTIVE\_DENSITY
& Active-row density signal. \\

LF\_FRAC\_VAR
& Fractional-variable signal. \\

LF\_REPAIR\_NCLOSE0
& Numerical repair-close-to-zero signal. \\

LF\_BASIC\_COUNT
& Basic-variable count signal. \\

LF\_FACE\_L2
& Displacement from the base state on the candidate face. \\
\hline
\end{tabular}}
{}
\end{table}

\newpage

\section{Audit and Reproducibility}
\label{app:audit-reproducibility}

\subsection{Validation Status Counts}
\label{app:validation-status-counts}

\begin{table}[htbp]
\TABLE
{Logical Validation Status Counts\label{tab:logical-validation-status-counts}}
{\scriptsize
\setlength{\tabcolsep}{5pt}
\begin{tabular}{@{}llr@{}}
\hline\up
Mode & Logical status & Count \\
\hline\up
B  & OPTIMAL   & 3,792 \\
B  & TIMELIMIT & 2,699 \\
B  & OTHER     & 1 \\
S  & OPTIMAL   & 3,812 \\
S  & TIMELIMIT & 2,661 \\
S  & OTHER     & 0 \\
R5 & OPTIMAL   & 3,894 \\
R5 & TIMELIMIT & 2,604 \\
R5 & OTHER     & 0 \down\\
\hline
\end{tabular}}
{\(R5\) counts are logical run outcomes after child-terminal and
parent-interruption reconciliation.}
\end{table}

\subsection{Remaining Missing Rows by Mode}
\label{app:remaining-missing-rows-by-mode}

\begin{table}[htbp]
\TABLE
{Validation Missing Rows by Mode\label{tab:validation-missing-rows-by-mode}}
{\scriptsize
\setlength{\tabcolsep}{5pt}
\begin{tabular}{@{}lrrrr@{}}
\hline\up
Mode
& Expected rows
& Parsed rows
& Remaining missing
& Parsed fraction \\
\hline\up
B  & 6,500 & 6,492 & 8  & 0.9988 \\
S  & 6,500 & 6,473 & 27 & 0.9958 \\
R5 & 6,500 & 6,498 & 2  & 0.9997 \down\\
\hline
\end{tabular}}
{Expected rows correspond to 130 instances and 50 seeds per mode.
Missing rows are not imputed in the validation analysis.}
\end{table}

\subsection{Software, Hardware, and Resource Accounting}
\label{app:software-hardware-resource-accounting}

\begin{table}[htbp]
\TABLE
{Computational Environment and Resource Model\label{tab:computational-environment}}
{\scriptsize
\setlength{\tabcolsep}{5pt}
\begin{tabular}{@{}ll@{}}
\hline\up
Item & Value \\
\hline\up
SCIP version/build & SCIP 9.2.1, Release \\
SoPlex version/build & SoPlex 7.1.3, Release \\
Operating system & Linux, Ubuntu \\
Cluster/hardware & Seneca HPC, \texttt{cac\_cpu} \\
Internal solver threads & 1 \\
Time limit & 7200 seconds \\
B allocated CPUs/task & 1 \\
S allocated CPUs/task & 1 \\
L10 allocated CPUs/task & 6 \\
R5 allocated CPUs/task & 6 \\
R5 process structure & Five candidate children plus one parent/controller \\
Random seed pairing & Identical instance-seed pairs across compared modes
\down\\
\hline
\end{tabular}}
{\(R5\) uses five candidate child processes and one parent/controller under
the six-CPU allocation. Reported \(R5\) node and LP-iteration summaries refer
to the winning child rather than the sum across concurrent portfolio members.}
\end{table}

\subsection{Replication Materials}
\label{app:replication-materials}

The replication package contains the solver configuration files, instance and seed manifests, run-level records, parsed analysis data, and scripts used to construct the tables and figures reported in the paper. A persistent repository link will be added once the archival package is public.

\newpage

\section{Validation Set and Result Summaries}
\label{app:full-instance-results}

This appendix provides compact summaries of the held-out validation set used in Section~\ref{sec:benchmark_results}. The paper reports family-level composition and validation outcomes rather than a full 130-instance result table.

The baseline tractability regimes are defined from observed baseline behavior
on the parsed validation runs. An instance is baseline-tractable if \(B\)
solves at least 90\% of its recorded seeds, transition if \(B\) solves more
than 10\% but less than 90\%, and stalled if \(B\) solves at most 10\%.
These regimes are used only to analyze validation behavior and are distinct
from the selection roles used when constructing the validation set.

\subsection{Validation Regimes by Family}
\label{app:validation-family-regimes}

Table~\ref{tab:validation-family-regime-counts} reports the distribution of
the 130 validation instances across instance families and baseline
tractability regimes.

\begin{table}[htbp]
\TABLE
{Validation Instances by Baseline Regime and Family
\label{tab:validation-family-regime-counts}}
{\scriptsize
\setlength{\tabcolsep}{4pt}
\begin{tabular}{lrrrr}
\hline\up
Family
& Baseline-tractable
& Transition
& Stalled
& Total \\
\hline\up
graphsym  & 7  & 3 & 4 & 14 \\
indnum    & 6  & 0 & 4 & 10 \\
marketML  & 7  & 0 & 3 & 10 \\
network   & 14 & 0 & 1 & 15 \\
neos      & 4  & 2 & 9 & 15 \\
other     & 0  & 1 & 7 & 8 \\
packprod  & 13 & 3 & 7 & 23 \\
sched     & 11 & 1 & 8 & 20 \\
setassign & 6  & 2 & 7 & 15\down\\
\hline
Total     & 68 & 12 & 50 & 130 \\
\hline
\end{tabular}}
{Regimes are defined from the observed baseline solve rate on the parsed
validation runs.}
\end{table}

\subsection{Family-Level Validation Summary}
\label{app:validation-family-summary}

Table~\ref{tab:validation-family-summary-table} summarizes validation outcomes
by family. Optimal-completion counts are aggregated over recorded validation
runs. WVS quantities are computed at the instance level and then summarized
within each family.

\begin{table}[htbp]
\TABLE
{Family-Level Validation Summary
\label{tab:validation-family-summary-table}}
{\scriptsize
\setlength{\tabcolsep}{3pt}
\begin{tabular}{lrrrrrrrr}
\hline\up
Family
& Inst.
& B opt.
& S opt.
& R5 opt.
& R5-B
& Median B WVS
& Median R5 WVS
& Median R5/B WVS \\
\hline\up
graphsym  & 14 & 412 & 416 & 427 & +15 & 0.232 & 0.241 & 1.055 \\
indnum    & 10 & 298 & 298 & 298 & +0  & 0.152 & 0.175 & 0.942 \\
marketML  & 10 & 354 & 353 & 376 & +22 & 0.136 & 0.162 & 1.116 \\
network   & 15 & 696 & 696 & 697 & +1  & 0.291 & 0.260 & 0.844 \\
neos      & 15 & 270 & 267 & 287 & +17 & 0.218 & 0.244 & 0.993 \\
other     & 8  & 29  & 29  & 36  & +7  & 0.171 & 0.192 & 1.120 \\
packprod  & 23 & 757 & 775 & 793 & +36 & 0.324 & 0.307 & 0.892 \\
sched     & 20 & 590 & 591 & 600 & +10 & 0.189 & 0.239 & 0.958 \\
setassign & 15 & 386 & 387 & 380 & -6  & 0.225 & 0.191 & 0.927\down\\
\hline
Total     & 130 & 3792 & 3812 & 3894 & +102 & - & - & - \\
\hline
\end{tabular}}
{Optimal-completion counts are aggregated over recorded validation runs within
each family. R5-B is the difference in observed optimal-run totals and is not
a paired seedwise difference. Median R5/B WVS is computed from per-instance
ratios and then summarized within family.}
\end{table}

\subsection{Selected Validation Diagnostics}
\label{app:validation_diagnostics}

Table~\ref{tab:validation-diagnostics} reports per-instance quantities for the
validation cases discussed explicitly in Section~\ref{sec:benchmark_results}.
The first three rows illustrate gains aligned across solve completion, WVS,
and recorded node count. The next two show cases where the maximum observed
node count increases under \(R5\). The final three are stalled cases in which
both modes solve no seeds; their lower recorded node counts under \(R5\)
therefore do not indicate improved solve completion.

\begin{table}[htbp]
\TABLE
{Selected Instance-Level Validation Diagnostics
\label{tab:validation-diagnostics}}
{\scriptsize
\setlength{\tabcolsep}{4pt}
\begin{tabular}{lrrrrrr}
\hline\up
Instance
& B opt.
& R5 opt.
& B WVS
& R5 WVS
& B max nodes
& R5 max nodes \\
\hline\up
\texttt{glass4}
& 49 & 50 & 0.581 & 0.362 & 3,884,554 & 492,775 \\
\texttt{neos-2624317-amur}
& 46 & 49 & 1.101 & 1.002 & 2,233,413 & 150,350 \\

\texttt{supportcase34}
& 34 & 50 & 0.663 & 0.477 & 31,370 & 10,144 \\
\texttt{gen-ip016}
& 0 & 0 & 0.078 & 0.124 & 41,428,151 & 48,248,381 \\
\texttt{pigeon-13}
& 19 & 30 & 0.427 & 0.592 & 26,829,014 & 28,380,088 \\
\texttt{markshare1}
& 0 & 0 & 0.019 & 0.167 & 38,307,612 & 1 \\
\texttt{pigeon-20}
& 0 & 0 & 0.111 & 0.184 & 11,076,719 & 1 \\
\texttt{assign1-5-8}
& 0 & 0 & 0.034 & 0.057 & 6,828,856 & 1\down\\
\hline
\end{tabular}}
{Optimal counts are observed validation-run totals for each instance and mode.
Node columns report the maximum observed recorded node count under each mode's
resource model. For \(R5\), these values do not represent the sum of nodes
processed across concurrent portfolio members.}
\end{table}

\newpage
\section{Pilot-Set Diagnostics}
\label{app:pilot-diagnostics}

\subsection{Candidate Availability Depth}
\label{app:candidate-abundance}

Section~\ref{sec:bestofk_audit} reports the best generated LP state for each
instance-seed pair. This appendix asks whether that best state is typically
isolated or whether several generated states are no worse than the baseline.

For an instance-seed pair, let \(\mathcal{C}\) denote the valid generated
candidate set. Let \(N_b\) denote the baseline branch-and-bound node count, and
let \(N_c\) denote the node count obtained from candidate \(c\in\mathcal{C}\).
We call a candidate \emph{baseline-nonworsening} if \(N_c \le N_b\), and define
\[
A=\left|\{c\in\mathcal{C}: N_c \le N_b\}\right|.
\]
Thus \(A\) is the number of generated LP states whose downstream node count is
no worse than the baseline for the same seed.

For each instance \(i\) and threshold \(r\in\{1,\ldots,10\}\), we compute the
fraction of seeds satisfying \(A\ge r\). Figure~\ref{fig:candidate-availability-depth}
reports the median of this fraction across the 20 pilot instances, together with
the 25th-75th percentile range. This appendix uses a node-count criterion only,
rather than the full lexicographic best-of-\(K\) rule from
Section~\ref{sec:bestofk_protocol}, because its purpose is to measure the depth
of candidate availability rather than to define a deployable selection policy.

\begin{figure}[htbp]
    \centering
    \includegraphics[width=0.92\textwidth]{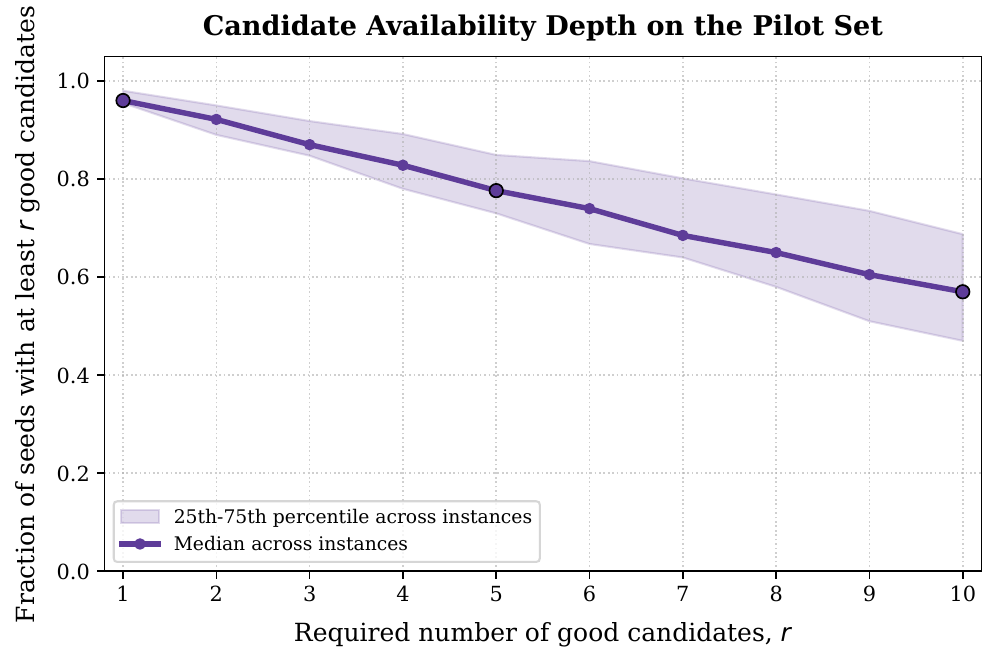}
    \caption{Candidate availability depth on the 20-instance pilot set.
    For each instance-seed pair, a generated LP state is counted as
    baseline-nonworsening if its branch-and-bound node count is no worse than
    the corresponding baseline run. For each threshold \(r\), the curve reports
    the median, across instances, of the fraction of seeds with at least \(r\)
    baseline-nonworsening candidates; the shaded band gives the 25th-75th
    percentile range across instances.}
    \label{fig:candidate-availability-depth}
\end{figure}

\subsection{LP-Iteration and Wall-Clock Profiles}
\label{app:pilot-component-profiles}

The main pilot analysis emphasizes WVS and the observed branch-and-bound
node-count tail. Figure~\ref{fig:pilot-component-profiles} reports companion
profiles for median LP iterations and median wall-clock time on the same
20-instance pilot set.

\begin{figure}[htbp]
    \centering
    \includegraphics[width=0.98\textwidth]
    {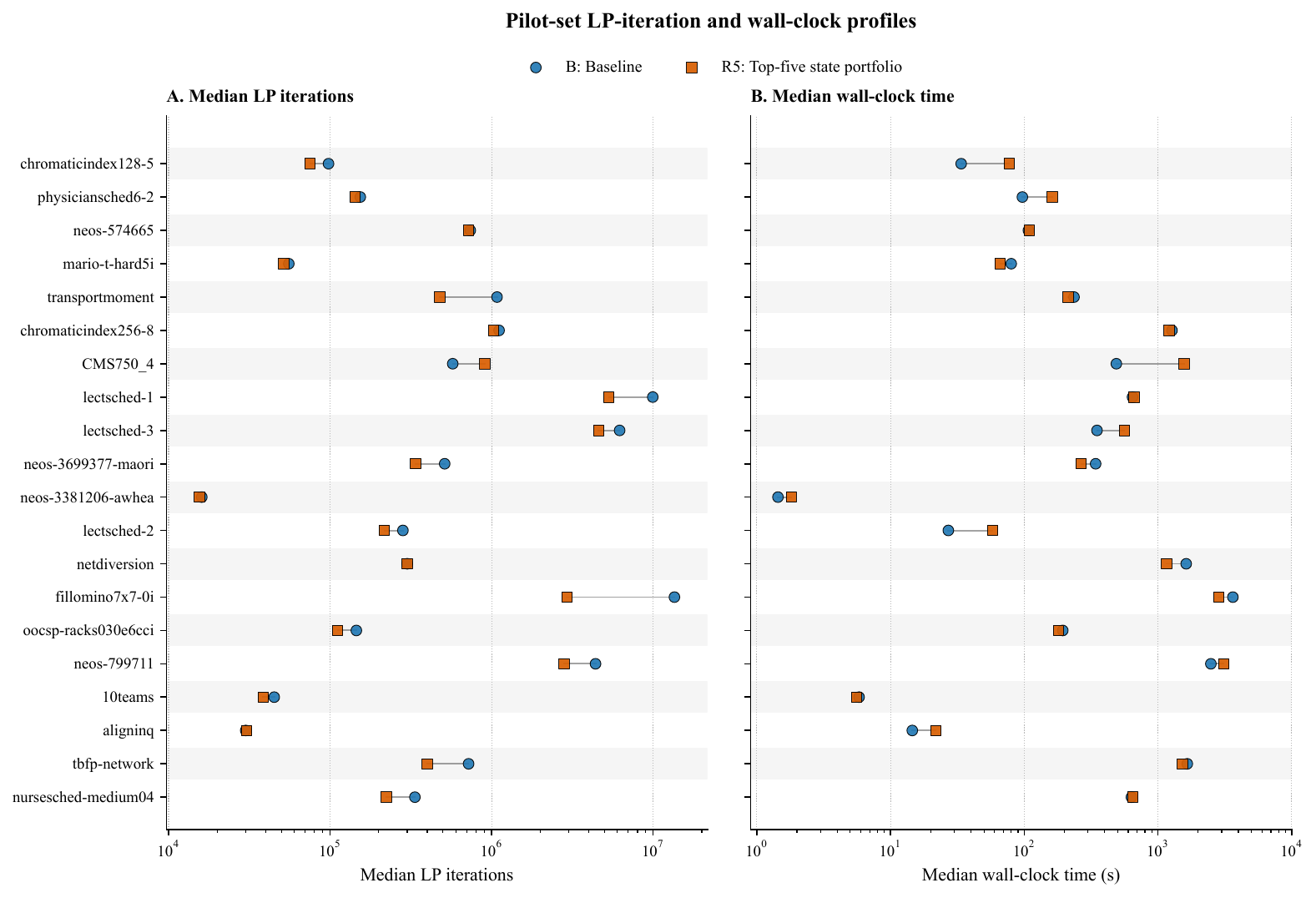}
    \caption{Pilot-set LP-iteration and wall-clock profiles for \(B\) and \(R5\). Panel A reports median LP iterations and Panel B median wall-clock time across the 100 paired seeds for each instance.}
    \label{fig:pilot-component-profiles}
\end{figure}
Instances follow the same baseline-WVS ordering as the main policy-recovery figure, and both horizontal axes use logarithmic scales. Mode \(B\) uses one CPU per task, whereas \(R5\) uses six. For \(R5\), the LP-iteration value characterizes the winning child that determines the run outcome rather than the sum of LP iterations across the concurrent portfolio.

\end{APPENDICES}

\ACKNOWLEDGMENT{We thank Professor David Shmoys and the Center for Data Science for Enterprise and Society at Cornell University for supporting access to the cloud-computing resources used in this study. We are also grateful to O.R. Advances through Collaboration (ORACL) for partial funding of this research.}







\end{document}